\documentclass[a4paper,reqno]{amsart}
\usepackage{amssymb}
\usepackage{amscd}
\usepackage{amsthm}
\usepackage{amsmath}
\usepackage{latexsym}
\usepackage{dsfont}
\usepackage{url}
\usepackage[all]{xy}
\usepackage{hyperref}

\theoremstyle{plain}
\newtheorem{theorem}{Theorem}
\newtheorem{corollary}[theorem]{Corollary}
\newtheorem{lemma}[theorem]{Lemma}
\newtheorem{proposition}[theorem]{Proposition}

\newtheorem{question}[theorem]{Question}

\theoremstyle{definition}
\newtheorem{definition}[theorem]{Definition}
\newtheorem{notation}[theorem]{Notation}
\newtheorem{example}[theorem]{Example}

\theoremstyle{remark}
\newtheorem{hypothesis}[theorem]{Hypothesis}
\newtheorem{remark}[theorem]{Remark}

\numberwithin{theorem}{section}

\newenvironment{ii}
{ \begin{enumerate}}
{\end{enumerate}}

\newcommand{\Spec}[1]{\mathrm{Spec}(#1)}

\newcommand{\Add}{\mbox{\rm{Add\,}}}

\newcommand{\add}{\mbox{\rm{add\,}}}
\newcommand{\ext}{\mbox{\rm{ext\,}}}

\newcommand{\Img}{\mbox{\rm{Im\,}}}

\newcommand{\Hom}[3]{\mathrm{Hom}_{#1}(#2,#3)}
\newcommand{\Ext}[4]{\mathrm{Ext}^{#1}_{#2}(#3,#4)}

\newcommand{\Der}[1]{\mathbf{D}({#1})}

\newcommand{\Cpx}[1]{\mathbf{C}({#1})} 
 
\newcommand{\Dle}[1]{\mathbf{D}^{\le{#1}}}
\newcommand{\Dge}[1]{\mathbf{D}^{\ge{#1}}}
\newcommand{\Ho}[1]{\mathrm{Ho}({#1})}

\newcommand{\DifA}{\mathbf{Dif}\A}
\newcommand{\GraA}{\mathbf{G}\A}
\newcommand{\CpxA}{\mathbf{C}\A} 
\newcommand{\DerA}{\mathbf{D}\A}

\newcommand{\rfmod}[1]{\mathrm{mod}\textrm{-}{#1}}
\newcommand{\rmod}[1]{\mathrm{Mod}\textrm{-}{#1}}

\newcommand{\ModR}{\rmod R}
\newcommand{\ModRR}{\rmod\R}
\newcommand{\Filt}[1]{\mathrm{Filt}\textrm{-}{#1}}
\newcommand{\Ab}{\mathrm{Ab}}

\newcommand{\pd}[2]{\mathrm{proj.dim}_{#1}{#2}}

\newcommand{\Supp}[1]{\mbox{\rm{Supp}} \,#1}
\newcommand{\Ker}[1]{\mbox{\rm{Ker}}\,#1}
\newcommand{\Coker}{\mbox{\rm{Coker}}}

\newcommand{\inv}{^{-1}}
\newcommand{\op}{\mathrm{op}}
\DeclareMathOperator{\hocolim}{hocolim}

\newcommand{\p}{\mathfrak{p}}
\newcommand{\q}{\mathfrak{q}}

\newcommand{\unit}{\mathds{1}}
\newcommand{\bbD}{\mathbb{D}}
\newcommand{\Z}{\mathbb{Z}}

\newcommand{\A}{\mathcal{A}}
\newcommand{\B}{\mathcal{B}}
\newcommand{\C}{\mathcal{C}}

\newcommand{\E}{\mathcal{E}}
\newcommand{\F}{\mathcal{F}}
\newcommand{\G}{\mathcal{G}}
\newcommand{\I}{\mathcal{I}}
\newcommand{\K}{\mathcal{K}}

\newcommand{\clP}{\mathcal{P}}
\newcommand{\R}{\mathcal{R}}
\newcommand{\clS}{\mathcal{S}}
\newcommand{\T}{\mathcal{T}}
\newcommand{\U}{\mathcal{U}}
\newcommand{\V}{\mathcal{V}}
\newcommand{\W}{\mathcal{W}}
\newcommand{\X}{\mathcal{X}}
\newcommand{\Y}{\mathcal{Y}}

\newcommand{\RHom}[3]{\mbox{\rm{{\bf R}Hom}}_{#1}(#2,#3)}
\newcommand{\LOtimes}[3]{#2 \otimes_{#1}^{\mathbf{L}} #3}

\newcommand{\dd}{\colon}
\newcommand{\la}{\longrightarrow}

\renewcommand{\iff}{if and only if }
\newcommand{\st}{such that }
\newcommand{\wrt}{with respect to }

\begin{document}

\title[Compactly generated pairs and co-$t$-structures]%
{On compactly generated torsion pairs and the classification of co-$t$-structures for commutative noetherian rings}

\author{Jan \v{S}\v{t}ov\'\i\v{c}ek} 
\address{Charles University, Faculty of Mathematics and Physics, Department of Algebra \\
Sokolovsk\'{a} 83, 186 75 Prague 8, Czech Republic}
\email{stovicek@karlin.mff.cuni.cz}

\author{David Posp{\'\i}{\v s}il} 
\email{pospisil.david@gmail.com}

\date{\today}
\thanks{This research was supported by GA~\v{C}R P201/12/G028.}

\begin{abstract}
We classify compactly generated co-$t$-structures on the derived category of a commutative noetherian ring. In order to accomplish that, we develop a theory for compactly generated Hom-orthogonal pairs (also known as torsion pairs in the literature) in triangulated categories that resembles Bousfield localization theory. Finally, we show that the category of perfect complexes over a connected commutative noetherian ring admits only the trivial co-$t$-structures and (de)suspensions of the canonical co-$t$-structure and use this to describe all silting objects in the category.
\end{abstract}

\maketitle

\setcounter{tocdepth}{1}
\tableofcontents

\section*{Introduction}

The aim of this paper is to give a classification of compactly generated co-$t$-structures on the unbounded derived category of a commutative noetherian ring. Along the way, we have found it useful to develop an ``unstable'' analogue of Bousfield localization theory for triangulated categories which rendered the attempts to classify compactly generated $t$-structures on one hand and compactly generated co-$t$-structures on the other hand as exactly the same problem. Thus, we do not have to start from scratch, but we can use existing results on $t$-structures. In our case this refers to the work of Alonso, Jerem{\'{\i}}as and Saor{\'{\i}}n~\cite{AJS}.

\smallskip

Co-$t$-structures have been introduced independently by Pauksztello~\cite{Pauk08} and Bondarko~\cite{Bond10}. Pauksztello was motivated by the analogies between rational homotopy theory and homological methods in commutative algebra~\cite{AH86} and his goal was to find a version of results on $t$-structures from~\cite{HKM02} which would work for the derived categories of cochain dg algebras. Bondarko's motivation~\cite{BondSurvey,Bond10} stemmed from problems related to Voevodsky's motives~\cite{BV08}.

Since then the concept turned useful in other contexts. A bijection between bounded co-$t$-structures on a small triangulated category and so-called silting subcategories has been established in~\cite{Bond10,MSSS11}. Silting subcategories are important representation theoretic objects which originated in~\cite{KV88} and have been studied in connection with $t$-structures and various applications in~\cite{AI12,HKM02,KY12}. They also occurred in connection with constructing geometric invariants of triangulated categories~\cite{JoPa11} and in other situations; see~\cite{Pauk11}.

In order to understand co-$t$-structures better, it seems helpful to provide more examples or, as is our case, a classification for a large class of triangulated categories. In general one has to be careful since even in~$\Der\Ab$, the unbounded derived category of abelian groups, there exists a proper class of co-$t$-structures. Indeed, there is a proper class of complete cotorsion pairs in $\Ab$ by~\cite[Example 2.2.2]{GT} and it is not difficult to produce an injective correspondence assigning to each such pair a co-$t$-structure in $\Der\Ab$. Hence we restrict our efforts to the problem of classifying compactly generated co-$t$-structures.

One readily notices analogies with Bousfield localizations, for which the classification~\cite{Nee92} in the case of commutative noetherian rings has been available for two decades. A Bousfield localization of a triangulated category is simply a Verdier localization $q^*\dd \T \to \T/\A$ which admits a right adjoint $q_*\dd \T/\A \to \T$. The important observation made by Bousfield and worked out abstractly for instance in~\cite{HPS97,Krause10,Nee01}, is that such a localization is fully described by the pair of classes $(\A,\B) = (\Ker q^*,\Img q_*)$ satisfying the following properties: $\A,\B$ are closed under suspensions and desuspensions, $\T(\A,\B) = 0$, and each $X \in \T$ admits an approximation triangle $A \to X \to B \to \Sigma A$. Such pairs are also called semi-orthogonal decompositions in the context of algebraic geometry~\cite{BoOr95}. Of special interest in our case are compactly generated (sometimes also called finite) localizations. We may start with a set of compact objects $\clS \subseteq \T$ in a triangulated category with coproducts and consider the Bousfield localization $q^*$ which is universal \wrt the property that $q^*|_\clS = 0$. It is well known in this case that:
\begin{ii}
\item The compact objects which are annihilated by $q^*$ are precisely those in the thick subcategory generated by $\clS$.
\item If $\T$ is nice enough, say well-generated in the sense of~\cite{Nee01}, we are in the recollement situation: $q_*$ has a right adjoint $q^!$ and the pair $(\B,\Y) = (\Img q_*,\Ker q^!)$ gives another Bousfield localization.
\end{ii}

A considerable part of the present paper is spent on proving analogs of these for torsion pairs or, as we call them, complete Hom-orthogonal pairs in triangulated categories. These are simply pairs $(\A,\B)$ \st $\T(\A,\B) = 0$ and each $X \in \T$ admits an approximation triangle $A \to X \to B \to \Sigma A$, dropping any assumption on $\A$ or $\B$ being closed under (de)suspensions. This makes sense since co-$t$-structures are special instances of such pairs, and our arguments would not become simpler if we specialized to co-$t$-structures. Interestingly, although the simple minded analogs of (i) and (ii) do hold under suitable and rather mild assumptions, the proofs become very technical.

Returning to our main classification goal, we prove that if $\T$ is at the base of a stable derivator and $(\A,\B)$ is generated by a set $\clS$ of compacts, then all other compacts in $\A$ are obtained from $\clS$ by taking extensions and summands. This generalizes (i) and implies among others that classifying compactly generated co-$t$-structures in $\T$ is the same as classifying right preaisles in the category of compact objects. Using an appropriate duality, we may classify left preaisles instead, which corresponds to the above mentioned classification of compactly generated $t$-structures.

Our generalization of (ii) works for all compactly generated algebraic triangulated categories and, although not directly related to the classification, may be of its own interest. It says that a compactly generated Hom-orthogonal pair $(\A,\B)$ always has a right adjacent complete Hom-orthogonal pair $(\B,\Y)$. As recollements proved useful in algebraic geometry and elsewhere, also adjacent pairs of $t$-structures and co-$t$-structures have their significance in studying weight filtrations of functors~\cite{BondSurvey,Bond10}.

\smallskip
 
We conclude the introduction by briefly outlining how the paper is organized. In Section~\ref{sec:approx} we recall necessary results from the approximation theory and the theory of cotorsion pairs for exact categories with transfinite compositions of inflations. We use this later for proving the generalization of the recollement situation. In Section~\ref{sec:triang} we discuss technical conditions on triangulated categories. In the first part we prove some basic results on stable derivators which we have not been able to find in the literature. We show that diagrams of global bicartesian squares are homotopy cartesian and establish a weak exactness of countable homotopy colimits in the spirit of~\cite{KeNi11}. In the second part we provide some details on models for compactly generated algebraic triangulated categories which we need later. In Section~\ref{sec:comp_gen} we prove the above mentioned generalizations of (i) and (ii) for compactly generated Hom-orthogonal pairs. In Section~\ref{sec:co-t-str} we obtain the classification of compactly generated co-$t$-structures on the derived category of a commutative noetherian ring and finally in Section~\ref{sec:compact} we determine which of these restrict to co-$t$-structures on the category of perfect complexes.

\section{Approximations and cotorsion pairs}
\label{sec:approx}

\subsection{Approximations}
\label{subsec:approx}

We start with an abstract general concept of approximations.

\begin{definition} \label{def:approx}
Let $\G$ be a category and $\F \subseteq \G$ a full subcategory. A morphism $f\dd F \to X$ in $\G$ with $F \in \F$ is called a \emph{right $\F$-approximation} of $X$ if any other morphism $f'\dd F' \to X$ with $F' \in \F$ factors through $f$. Dually, $g\dd X \to F$ with $F \in \F$ is called a \emph{left $\F$-approximation} of $X$ if every $g'\dd X \to F'$ with $F' \in \F$ factors through $g$.
\end{definition}

\subsection{Exact categories}
\label{subsec:exact}

In order to present existence theorems for left or right approximations, we restrict to the case where $\G$ is an exact category. The concept is originally due to Quillen, but the common reference for a simple axiomatic description is~\cite[Appendix A]{Keller90} and an extensive treatment is given in~\cite{Bueh10}.

An \emph{exact category} is an additive category $\E$ together with a distinguished class of diagrams of the form
\[ 0 \la X \overset{i}\la Y \overset{d}\la Z \la 0, \]
called \emph{conflations}, satisfying a certain collection of axioms which makes conflations behave similar to short exact sequences in an abelian category and allows to define Yoneda Ext groups with usual properties.

Adopting the terminology from~\cite{Keller90}, the second map in a conflation (denoted above by $i$) is called \emph{inflation}, while the third map (denoted by $d$) is referred to as \emph{deflation}.

All what we need to know about exact categories for the purpose of this text is summarized in the following proposition:

\begin{proposition} \label{prop:exact_categories} \cite{Bueh10,Keller90} ~
\begin{ii}
\item Let $\A$ be an abelian category and $\G \subseteq \A$ be an extension closed subcategory. Then $\G$, considered together with all short exact sequences in $\A$ whose all terms belong to $\G$, is an exact category.
\item Every small exact category arises up to equivalence as in (i).
\end{ii}
\end{proposition}

In our situation, we are concerned with exact categories which are not small, but still arise as in Proposition~\ref{prop:exact_categories}(i). In fact, one easily sees that Proposition~\ref{prop:exact_categories}(i) holds in the following more general form:

\begin{lemma} \label{lem:ext_subcat} \cite[Lemma 10.20]{Bueh10}
Let $\G$ be an exact category and $\F$ be a full subcategory which is closed under extensions. Then $\F$ together with all conflations in $\G$ whose all terms belong to $\F$, is an exact category.
\end{lemma}

\subsection{Cotorsion pairs}
\label{subsec:cotorsion}

An important source of approximations in exact categories are so-called cotorsion pairs.

\begin{definition} \label{def:cotorsion}
A pair of subcategories $(\A,\B)$ of an exact category $\G$ is said to be a \emph{cotorsion pair} if
\begin{align*}
\A &= \{ A \in \G \mid \Ext1\G AB = 0 \textrm{ for each } B \in \B \}, \\
\B &= \{ B \in \G \mid \Ext1\G AB = 0 \textrm{ for each } A \in \A \}.
\end{align*}

A cotorsion pair $(\A,\B)$ is \emph{complete} if for each $X \in \G$ there exist conflations
\[
0 \to B_X \to A_X \overset{f_X}\to X \to 0
\quad \textrm{and} \quad
0 \to X \overset{g^X}\to B^X \to A^X \to 0
\]
\st $A_X, A^X \in \A$ and $B_X, B^X \in \B$. We will call such conflations \emph{approximation sequences} for $X$.

The cotorsion pair $(\A,\B)$ is \emph{functorially complete} if the approximation sequences can be made functorial in~$X$.
\end{definition}

Approximation sequences for given $X \in \G$ as above are almost never unique, and neither are their functorial versions. To justify the terminology, it is an easy observation that $f_X$ is a right $\A$-approximation of $X$, while $g^X$ is a left $\B$-approximation of $X$.

\subsection{Filtrations and deconstructibility}
\label{subsec:deconstr}

Before discussing how one obtains complete cotorsion pairs, we need to define the concepts of transfinite composition, filtration and deconstructibility.

\begin{definition} \label{def:t_compos} \cite[\S2.1.1]{Hov99}, \cite[\S10.2]{Hirsch03}
Let $\I$ be a class of morphisms in a category~$\G$.

\begin{enumerate}
\item If $\lambda$ is an ordinal number, a \emph{$\lambda$-sequence of maps in $\I$} is a well ordered direct system $(X_\alpha, i_{\alpha\beta} \mid \alpha < \beta < \lambda)$ in $\G$ indexed by $\lambda$ \st
\begin{ii}
  \item For each limit ordinal $\mu < \lambda$, the cocone $(X_\mu, i_{\alpha\mu} \mid \alpha<\mu)$ of the subsystem $(X_\alpha, i_{\alpha\beta} \mid \alpha < \beta < \mu)$ is a colimit cocone. In other words, $X_\mu \cong \varinjlim_{\alpha<\mu} X_\alpha$ canonically;

  \item The morphism $i_{\alpha,\alpha+1}\dd X_\alpha \to X_{\alpha+1}$ belongs to $\I$ for each $\alpha+1<\lambda$. 
\end{ii}

\item The \emph{composition} of a $\lambda$-sequence $(X_\alpha, i_{\alpha\beta} \mid \alpha < \beta < \lambda)$ is the colimit map $X \to \varinjlim_{\alpha<\lambda} X_\alpha$, if the colimit exists.

\item A \emph{transfinite composition} of maps in $\I$ is a map in $\G$ that is the composition of a $\lambda$-sequence of maps in $\I$.
\end{enumerate}
\end{definition}

Our definition of filtration generalizes the corresponding concept in module theory from~\cite[Definition 3.1.1]{GT}.

\begin{definition} \label{def:filtration} \cite[Definition 2.9]{SaSt11}
Let $\G$ be an exact category and $\clS$ be a class of objects of $\G$. An object $X \in \G$ is said to be \emph{$\clS$-filtered} if $0 \to X$ is a transfinite composition of inflations with cokernel in $\clS$. That is, we require the existence of a well-ordered direct system $(X_\alpha, i_{\alpha\beta} \mid \alpha < \beta \le \lambda)$ \st the following holds:
\begin{ii}
\item $X_0 = 0$ and $X_\lambda = X$,

\item $X_\mu \cong \varinjlim_{\alpha<\mu} X_\alpha$ canonically for each limit ordinal $\mu<\lambda$,

\item For each $\alpha+1 \le \lambda$, there is a conflation in $\G$ of the form
\[ 0 \la X_\alpha \overset{i_{\alpha,\alpha+1}}\la X_{\alpha+1} \la S_\alpha \la 0 \]
with $S_\alpha \in \clS$.
\end{ii}

The direct system $(X_\alpha, i_{\alpha\beta} \mid \alpha < \beta \le \lambda)$ is then called an \emph{$\clS$-filtration} of $X$. The class of all $\clS$-filtered objects in $\G$ is denoted by $\Filt\clS$.
\end{definition}

Finally we arrive at the concept of deconstructibility.

\begin{definition} \label{def:deconstructible}
Let $\G$ be an exact category and $\F \subseteq \G$ be a class of objects. Then $\F$ is called \emph{deconstructible} if there exists a set (not a proper class!) $\clS \subseteq \G$ of objects \st $\F = \Filt\clS$.
\end{definition}

At this point we need to impose more assumptions on our exact categories. Despite the fact that the definition above works for arbitrary exact categories, in order to construct complete cotorsion pairs we need to specialize.

\begin{hypothesis} \label{hyp:exact}
Let $\G$ be an exact category. We adopt the assumptions on $\G$ from~\cite[Setup 1.1]{SaSt11}:
\begin{ii}
\item Arbitrary transfinite compositions of inflations exist and are themselves inflations.
\item Every object of $\G$ is small relative to the class of all inflations in the sense of~\cite[Definition 2.1.3]{Hov99}, or equivalently of~\cite[Definition 10.4.1]{Hirsch03}. That is, for every object of $Y \in \G$, there exists a cardinal $\kappa = \kappa(Y)$ \st for every regular cardinal $\lambda \ge \kappa$ and every $\lambda$-sequence of inflations $(X_\alpha, i_{\alpha\beta} \mid \alpha < \beta < \lambda)$, the canonical morphism of abelian groups
\[ \varinjlim \G(Y,X_\alpha) \la \G(Y, \varinjlim X_\alpha) \]
is an isomorphism.
\end{ii}
\end{hypothesis}

Condition (i) above asserts a weak version of left exactness for colimits, and it implies the existence and exactness of coproducts; see~\cite[Lemma 1.4]{SaSt11}. Condition (ii) is a technical one and is needed for the small object argument as applied in~\cite[Proposition 2.1]{SaSt11}. We shall exhibit two broad classes of exact categories satisfying Hypothesis~\ref{hyp:exact} which are relevant for the present paper:

\begin{example} \label{expl:efficient-mod}
If $R$ is a unital and associative (but not necessarily commutative) ring, $\ModR$ will stand for the category of all right $R$-modules. Then $\ModR$ with the abelian exact structure (that is, conflations are precisely short exact sequences) meets Hypothesis~\ref{hyp:exact}(i) and~(ii).

More generally, given a small preadditive category $\R$, we will denote by $\ModRR$ the category of all right $\R$-modules. That is, the objects of $\ModRR$ are additive functors
\[ \R^\op \la \Ab \]
and the morphisms are natural transformations. Then $\ModRR$ satisfies Hypothesis~\ref{hyp:exact}.

In fact yet more generally, any Grothendieck category with the abelian exact structure satisfies Hypothesis~\ref{hyp:exact}.
\end{example}

\begin{example} \label{expl:efficient-cpx}
Let $\Cpx\ModRR$ be the category of cochain complexes with the componentwise split exact structure. Then $\Cpx\ModRR$ satisfies Hypothesis~\ref{hyp:exact}; see~\cite[Example 2.8(2)]{SaSt11}. A more general situation of this type involving dg modules over small dg categories will be discussed in Example~\ref{expl:dg_modules_ring} and Remark~\ref{rem:cofibrant}.
\end{example}

A few useful properties of deconstructible classes in exact categories under our hypotheses are summarized in the following lemma.

\begin{lemma} \label{lem:deconstr_basic}
Let $\G$ be an exact category satisfying Hypothesis~\ref{hyp:exact}. Then the following hold:

\begin{ii}
\item If $\F \subseteq \G$ is deconstructible, then $\F$ is closed under extensions and coproducts, so $\F$ is naturally an exact category in the sense of Lemma~\ref{lem:ext_subcat}. Moreover, $\F$ itself satisfies Hypothesis~\ref{hyp:exact} and the inclusion $\F \to \G$ preserves $\lambda$-sequences of inflations and their compositions.

\item Let $\G' \subseteq \G$ be an extension closed full subcategory with the induced exact structure. Suppose further that $\G'$ satisfies Hypothesis~\ref{hyp:exact} and the inclusion $\G' \to \G$ preserves $\lambda$-sequences of inflations and their compositions. Then $\F \subseteq \G'$ is deconstructible in $\G'$ \iff it is deconstructible in $\G$. 

\item Deconstructibility is transitive. That is, if $\F' \subseteq \F \subseteq \G$ \st $\F$ is deconstructible in $\G$ and $\F'$ is deconstructible in $\F$, then $\F'$ is deconstructible in $\G$.
\end{ii}
\end{lemma}

\begin{proof}
(i) By~\cite[Corollary 2.11]{SaSt11}, any $\F$-filtered object belongs to $\F$. It follows trivially that $\F$ is closed under extensions and it is closed under coproducts by the proof of~\cite[Lemma 1.4(1)]{SaSt11}. Suppose now that we have a $\lambda$-sequence $(X_\alpha, i_{\alpha\beta} \mid \alpha < \beta < \lambda)$ of inflations in $\F$. We will prove the following jointly by induction on $\lambda>0$:
\begin{enumerate}
\item[(a)] $(X_\alpha, i_{\alpha\beta} \mid \alpha < \beta < \lambda)$ is a $\lambda$-sequence of inflations in $\G$.
\item[(b)] The composition $f\dd X_0 \to X$ of the $\lambda$-sequence in $\G$ is an  inflation with cokernel in $\F$.
\end{enumerate}
Clearly, (a) and (b) are true for $\lambda = 1$. Note also that if $\lambda$ is arbitrary and (b) holds, then $X \in \F$ since $\F$ is extension closed. In particular, $f\dd X_0 \to X$ is a composition of the $\lambda$-sequence in $\F$ and it is an inflation in $\F$.

Suppose now that we have proved (a) and (b) for all $\mu < \lambda$. Hence given any limit ordinal $\mu<\lambda$, then $(X_\alpha, i_{\alpha\beta} \mid \alpha < \beta < \mu)$ is a $\lambda$-sequence of inflations in $\G$ whose composition in $\G$ is the same as in $\F$: the map $i_{0\mu}\dd X_0 \to X_\mu$. Thus (a) follows for $\lambda$. Moreover, the cokernel of the composition $f\dd X_0 \to X$ of $(X_\alpha, i_{\alpha\beta} \mid \alpha < \beta < \lambda)$ in $\G$ is filtered by the cokernels of $i_{\alpha,\alpha+1}$; see~\cite[Lemma 3.5]{Bueh10}. Appealing to \cite[Corollary 2.11]{SaSt11} again, $\Coker f$ belongs to $\F$, which proves (b).

To summarize, we have proved that Hypothesis~\ref{hyp:exact}(i) holds for $\F$ and that $\F \to \G$ preserves $\lambda$-sequences of inflations and their compositions. Therefore, Hypothesis~\ref{hyp:exact}(ii) follows for $\F$ from the fact that it holds for $\G$.

(ii) Let $\clS$ be a set of objects in $\F$. Using a similar argument as in (i), one proves that $\F = \Filt\clS$ in $\G'$ \iff $\F = \Filt\clS$ in $\G$.

(iii) This is a straightforward consequence of~(i) and~(ii).
\end{proof}

We conclude our discussion of filtrations with an important property of deconstructible classes in Grothendieck categories, which we need later.

\begin{proposition} \label{prop:deconstr_intersect} \cite[Proposition 2.9]{Sto11}
Let $\G$ be a Grothendieck category, considered as an exact category with the abelian exact structure. Then:
\begin{ii}
\item The closure of a deconstructible class under direct summands is deconstructible.
\item The intersection $\bigcap_{i \in I} \F_i$ of a collection of deconstructible classes $(\F_i \mid i \in I)$, indexed by a set $I$, is deconstructible.
\end{ii}
\end{proposition}

\subsection{Existence of complete cotorsion pairs}
\label{subsec:complete}

Now we have introduced all the terminology which is necessary to present a result on the existence and the structure of certain complete cotorsion pairs. The version here comes from~\cite[Corollary 2.15]{SaSt11}, but there exist several other versions under various assumptions in the literature. We say that a class $\F \subseteq \G$ of objects of an exact category is \emph{generating} if every object $X \in \G$ admits a deflation of the form $\coprod_{i \in I} F_i \to X$ with all $F_i \in \F$.

\begin{proposition} \label{prop:completeness} \cite[Corollary 2.15(2) and Remark 3.8]{SaSt11}
Let $\G$ be an exact category as in Hypothesis~\ref{hyp:exact}, and let $\clS \subseteq \G$ be a set (not a proper class!) of objects \st $\Filt\clS$ is generating. Denote
\begin{align*}
\B &= \{ B \in \G \mid \Ext1\G SB = 0 \textrm{ for each } S \in \clS \}, \\
\A &= \{ A \in \G \mid \Ext1\G AB = 0 \textrm{ for each } B \in \B \}.
\end{align*}
Then the following hold:

\begin{ii}
\item $(\A,\B)$ is a functorially complete cotorsion pair in $\G$.
\item $X \in \G$ belongs to $\A$ \iff $X$ is a summand of an $\clS$-filtered object.
\end{ii}
\end{proposition}

\section{The hierarchy of triangulated categories}
\label{sec:triang}

Although the main aim of this paper is to study co-$t$-structures in the derived category of a commutative noetherian ring, some results are true for much more general triangulated categories. Here we summarize which triangulated categories we will be considering.

The most general object of our interest will be a triangulated category with small coproducts. We refer to~\cite{Nee01} for basic properties of such categories.

A more special but yet very general and abstract is the class of triangulated categories of the form $\T = \bbD(e)$, where $\bbD$ is a stable derivator in the sense of~\cite{CisNee08,Gr12,Malt07} and $e$ is the one-point category. We will discuss such categories in~\S\ref{subsec:derivators}.

Specializing further, we get the class of compactly generated algebraic triangulated categories in the sense of~\cite{KellerDG,KrauseChicago}. We devote \S\S\ref{subsec:algebraic} and~\ref{subsec:model} to basic properties of such categories.

\subsection{Stable derivators}
\label{subsec:derivators}

The theory of derivators was initiated by Grothendieck~\cite{GrothDerivators}, Heller~\cite{Heller88}, Keller~\cite{Keller91} and Franke~\cite{Franke96}, where the term derivator was coined by Grothend\-ieck. A motivation for defining derivators is an attempt to fix the unfortunate fact that diagram categories of a triangulated category need not be triangulated. Derivators circumvent this problem by including ``derived versions'' of the diagram categories in the structure. More historical notes and motivating points can be found in the introduction of~\cite{CisNee08} or~\cite{Gr12}.

A \emph{stable derivator} (also called \emph{triangulated derivator} in the literature) $\bbD$ is a strict $2$-functor
\[ Dia^\op \la CAT, \]
whose domain is a $2$-subcategory $Dia$ of the $2$-category $Cat$ of all small categories, and which satisfies certain axioms. The codomain of $\bbD$ is the ''$2$-category'' of all (not necessarily small) categories. In this paper, we will always assume that $Dia = Cat$ and we will only be interested in the case where $\bbD(I)$ is a category with small coproducts for every small category $I$. In fact, \cite{Gr12} includes the latter assumption already in the definition of a derivator; see~\cite[Proposition 1.7]{Gr12}.

We omit the precise lengthy definition of a stable derivator here. There are various good sources available where the reader might wish to look for the definition as well as for basic properties. Many of them are proved in~\cite{CisNee08}, while a nice and detailed treatment of the topic involving a simplification of the axioms is given in~\cite{Gr12}. Concise treatments can also be found in~\cite{Cis03,Malt07}.

A word of warning regarding the just mentioned references: There is some ambiguity in what $Cat^\op$ means for natural transformations. One may formally invert their direction as in~\cite{Cis03,CisNee08,KeNi11,Malt07}, but one may also keep their direction as in~\cite{Gr12}. This difference is inessential but one should keep it in mind. We will stick to the majority convention that the direction of $2$-morphisms is inverted as well. When referring to statements in~\cite{Gr12}, we will always mean their appropriately dualized versions.

Now suppose that $\bbD$ is a stable derivator, $e$ the terminal category in $Cat$ and $I$ is any small category. If we denote $\T = \bbD(e)$, then $\bbD(I)$ can be viewed as an approximation of the diagram category $\T^{I^\op}$ and there is always a comparison functor $d_I\dd \bbD(I) \to \T^{I^\op}$ (see~\cite[\S1.10]{CisNee08}). The favorable fact is that unlike $\T^{I^\op}$, the category $\bbD(I)$ \emph{is} canonically triangulated:

\begin{proposition} \label{prop:deriv_triang}
Let $\bbD\dd Cat^\op \to CAT$ be a stable derivator. Then:
\begin{ii}
\item $\bbD(I)$ is canonically triangulated for each $I \in Cat$;
\item Given any morphism $u\dd I \to J$ in Cat, then $u^*\dd \bbD(J) \to \bbD(I)$ is a triangulated functor. Here we denote $\bbD(u)$ by $u^*$ as is customary.
\end{ii}
\end{proposition}

\begin{proof}
The description of the triangulated structure and a sketch of the proof are given in~\cite[Th\'eor\`eme 1]{Malt07}, along with the promise to provide the details in a future paper. In the meantime, a thorough account has been given in~\cite[Theorem 4.16 and Corollary 4.19]{Gr12}.
\end{proof}

Next let us focus on the possible variants of homotopy pullbacks and pushouts in the context of triangulated categories and stable derivators. A version which works for \emph{any} triangulated category can be found for example in~\cite[Definition 1.4.1]{Nee01}:

\begin{definition} \label{def:htp_cartesian}
A commutative square in a triangulated category
\[
\begin{CD}
W @>{x}>> X       \\
@V{y}VV @VV{u}V   \\
Y @>{v}>> Z       \\
\end{CD}
\]
is called \emph{homotopy cartesian} if there exists a triangle of the form
\[
\begin{CD}
W @>{\begin{pmatrix}\phantom{-}x \\ -y\end{pmatrix}}>>
X \oplus Y @>{\begin{pmatrix}u & v\end{pmatrix}}>>
Z @>>> \Sigma W.
\end{CD}
\]
\end{definition}

Following the spirit of derivators, another definition seems more convenient if $\T = \bbD(e)$. In order to state it, we adopt the following notation from~\cite{Malt07}: For $n < \omega$, we denote $\Delta_n = \{0,1,\dots,n\}$ and consider $\Delta_n$ with the natural ordering. As customary, we will view $\Delta_n$ as well as any other posets in the argument of $\bbD$ as small categories.

\begin{definition} \label{def:bicartesian} \cite[Definitions 3.9 and 4.1]{Gr12}
Let $\bbD$ be a stable derivator. Denote $\Box = \Delta_1^\op \times \Delta_1^\op$, and let $i_\ulcorner\dd \ulcorner \to \Box$ and $i_\lrcorner\dd \lrcorner \to \Box$ denote the inclusions from the subposets
\[
\begin{CD}
(0,0) @<<< (0,1)   \\
@AAA       @.      \\
(1,0) @.        
\end{CD}
\qquad \textrm{and} \qquad
\begin{CD}
      @.   (0,1)\phantom{,}   \\
@.         @AAA               \\
(1,0) @<<< (1,1),
\end{CD}
\]
respectively. An object $A \in \bbD(\Box)$ is called \emph{bicartesian} if it belongs to the essential image of the functor ${(i_\ulcorner)}_!\dd \bbD(\ulcorner) \to \bbD(\Box)$ (which is left adjoint to ${(i_\ulcorner)}^* = \bbD(i_\ulcorner)\dd \bbD(\Box) \to \bbD(\ulcorner)$). Equivalently, we can dually define bicartesian squares as those in the essential image of ${(i_\lrcorner)}_*\dd \bbD(\lrcorner) \to \bbD(\Box)$.
\end{definition}

The coming proposition answers in an expected way the natural question of how the two definitions relate. As we were unable to find this basic fact in the literature\footnote{It was later pointed to us by Bernhard Keller and Denis-Charles Cisinski that a very similar result has been independently obtained in \cite[Theorem 6.1]{GPS14} by Groth, Ponto and Shulman at around the same time.}, we are including a full proof.

\begin{proposition} \label{prop:htp_vs_bicart}
Let $\bbD$ be a stable derivator and $\T = \bbD(e)$. If $A \in \bbD(\Box)$ is a bicartesian object, then its diagram
\[
d_\Box(A)\dd \qquad
\begin{CD}
A_{00} @>>> A_{01}   \\
@VVV       @VVV      \\
A_{10} @>>> A_{11}      
\end{CD}
\]
is a homotopy cartesian square in $\T$.
\end{proposition}

\begin{proof}
We denote $W = A_{00}$, $X = A_{01}$, $Y = A_{10}$ and $Z = A_{11}$, so that $d_\Box(A) \in \T^{\Box^\op}$ has the form:
\[
\begin{CD}
W @>{x}>> X       \\
@V{y}VV @VV{u}V   \\
Y @>{v}>> Z       \\
\end{CD}
\]

In order to prove that the latter diagram is homotopy cartesian in $\T$, we construct two objects from $A$ in other categories in the image of $\bbD$ and compare their diagrams.

First, consider the following full subposet $I$ of $\Delta_2^\op \times \Delta_2^\op$:
\[
I\dd \qquad
\begin{CD}
(0,0) @<<< (0,1) @<<< (0,2)   \\
@AAA       @AAA               \\
(1,0) @<<< (1,1)              \\
@AAA                          \\
(2,0),
\end{CD}
\]
along with the obvious embeddings $i_1\dd \Box = \Delta_1^\op \times \Delta_1^\op \to I$ and $i_2\dd I \to \Delta_2^\op \times \Delta_2^\op$. Then one of our objects of interest will be $B = (i_2)_!(i_1)_*(A)$, whose diagram is of the form
\begin{equation} \label{eqn:2x2square}
d_{\Delta_2^\op \times \Delta_2^\op}(B)\dd \qquad
\begin{CD}
W @>{x}>> X @>>> 0           \\
@V{y}VV @V{u}VV @VVV         \\
Y @>{v}>> Z @>{p}>> V        \\
@VVV    @V{q}VV @VV{r}V      \\
0   @>>>  U @>{s}>> \Sigma W,
\end{CD}
\end{equation}
where $\Sigma X$ is the suspension of $X$. Let us be a little more specific. The fact that we have zeros in the two corners follows from~\cite[Proposition 3.6]{Gr12}. We deduce that all the objects in $\bbD(\Box)$ induced by poset embeddings $\Box \to \Delta_2^\op \times \Delta_2^\op$ onto the ``small squares'' in $\Delta_2^\op \times \Delta_2^\op$ are bicartesian by~\cite[Propositions 3.10 and 4.6]{Gr12} (Maltsiniotis calls such a $B$ \emph{polycartesian} in~\cite[\S3]{Malt07}). Abusing the terminology, we will simply say that all small squares in the above diagram are bicartesian. It follows immediately from the definition of $\Sigma$ (see~\cite[Definition 3.16]{Gr12}) that the lower right corner is occupied by $\Sigma W$ up to a canonical isomorphism.

Second, consider the full subposet $J \subseteq \Delta_2^\op \times \Delta_2^\op \times \Delta_1^\op \times \Delta_1^\op$, whose Hasse diagram is as follows (to make the picture readable, we write the coordinates in the form $abcd$ instead of $(a,b,c,d)$):
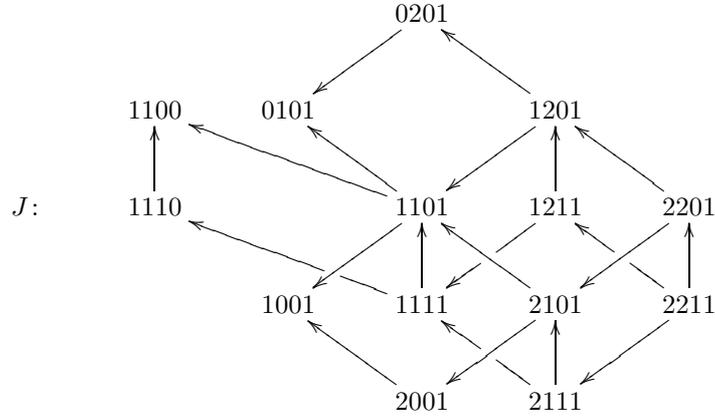
\[
\xymatrix{
& && 0201 \ar[dl] \\
& 1100 & 0101 && 1201 \ar[ul] \ar[dl] \\
J\dd & 1110 \ar[u] && 1101 \ar[ul] \ar[dl] \ar[ull] & 1211 \ar[u] \ar[dl]|\hole & 2201 \ar[ul] \ar[dl] \\
& & 1001 & 1111 \ar[u] \ar[ull]|(.335)\hole & 2101 \ar[ul] \ar[dl] & 2211 \ar[u] \ar[ul]|\hole \ar[dl] \\
& && 2001 \ar[ul] & 2111 \ar[u] \ar[ul]|\hole \\
}
\]
Consider also the full subposet $J_1 \subseteq J$ depicted below, along with the obvious embedding $j_1\dd J_1 \to J$:
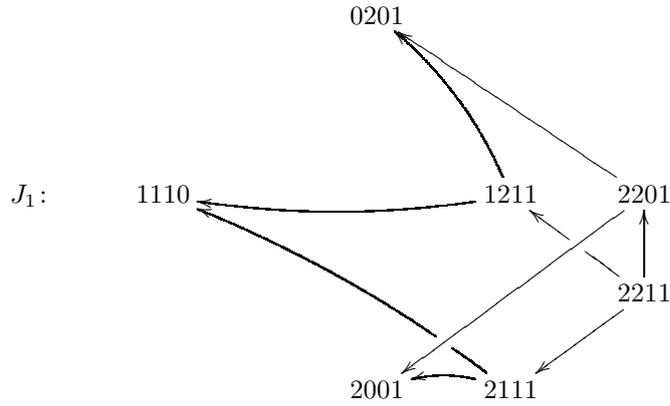
\[
\xymatrix{
& && 0201 \\
& &&& \\
J_1\dd & 1110 &&& 1211 \ar@/_/[uul] \ar@/^/[lll] & 2201 \ar[uull] \ar[ddll] \\
& &&&& 2211 \ar[u] \ar[ul]|\hole \ar[dl] \\
& && 2001 & 2111 \ar@/_/[l] \ar@/_/[uulll]|(.205)\hole \\
}
\]
Now we will construct an object $C \in \bbD(J)$ from $B \in \bbD(\Delta_2^\op \times \Delta_2^\op)$ as follows. We take the two embeddings $k\dd \lrcorner \to \Delta_2^\op \times \Delta_2^\op$ and $j_2\dd \lrcorner \to J_1$ given by
\begin{equation} \label{eqn:corner_maps}
\begin{matrix}
k\dd   &(1,0) &\mapsto &(2,1), &(0,1) &\mapsto &(1,2), &(1,1) &\mapsto &(2,2), \\
j_2\dd &(1,0) &\mapsto &2111,  &(0,1) &\mapsto &1211,  &(1,1) &\mapsto &2211.  \\
\end{matrix}
\end{equation}
Then the object $k^*(B)$ has the diagram of the shape
\[
\begin{CD}
     @.      V       \\
@.         @VV{r}V   \\
U @>{s}>> \Sigma W,
\end{CD}
\]
and we put $C = (j_1)_*(j_2)_!k^*(B) \in \bbD(J)$. We claim that the diagram of $C$ is up to isomorphism of the shape below, the maps among $X,Y,X\oplus Y$ are the product/coproduct maps, and all small squares are bicartesian.
\begin{equation} \label{eqn:4dim}
\begin{split}
\xymatrix{
& && 0 \ar@{<-}[dl] \\
& W & X && Y \ar@{<-}[ul] \ar@{<-}[dl] \\
d_J(C)\dd & 0 \ar@{<-}[u] && X \oplus Y \ar@{<-}[ul] \ar@{<-}[dl] \ar@{<-}[ull]^{\left(\begin{smallmatrix}\phantom{-}x \\ -y\end{smallmatrix}\right)} & V \ar@{<-}[u] \ar@{<-}[dl]|\hole_(.3){p\!} & 0 \ar@{<-}[ul] \ar@{<-}[dl] \\
& & Y & Z \ar@{<-}[u]|{(u,v)} \ar@{<-}[ull]|(.35)\hole & X \ar@{<-}[ul] \ar@{<-}[dl] & \Sigma W \ar@{<-}[u] \ar@{<-}[ul]|\hole_(.7){r} \ar@{<-}[dl]^{s} \\
& && 0 \ar@{<-}[ul] & U \ar@{<-}[u] \ar@{<-}[ul]|\hole^(.3){q\!} \\
}
\end{split}
\end{equation}
To see this, notice first that the occurrence of the zero objects in the diagram is a consequence of~\cite[Proposition 3.6]{Gr12}.

The bicartesian property of all squares can be obtained in a straightforward manner from~\cite[Propositions 3.10 and 4.6]{Gr12}. At that point it is also helpful to use the following observation: If $J_1 \subseteq J' \subseteq J$ and we factor correspondingly the inclusion $j_1\dd J_1 \to J$ as
\[ J_1 \overset{j'}\la J' \overset{j''}\la J, \]
then a commutative square which is contained in $J'$ is bicartesian in $C$ \iff it is bicartesian in $C' = (j')_*(j_2)_!k^*(B) \in \bbD(J')$. All we need here is to notice that canonically $C' \cong (j'')^*(C)$ by~\cite[Proposition 1.20]{Gr12}.

Let us focus on objects in $d_J(C)$ now. Comparing $d_J(C)$ with $d_{\Delta_2^\op \times \Delta_2^\op}(B)$ in (\ref{eqn:2x2square}) and using the fact that a bicartesian square is up to a canonical isomorphism determined by its restriction to $\ulcorner$ or $\lrcorner$ (cf.~\cite[Lemma 1.21]{Gr12}), one easily sees that the objects agree with our claim up to isomorphism. The only part of the diagram which may deserve some attention is the square
\[
\begin{CD}
X \oplus Y @>>> Y\phantom{,}  \\
@VVV          @VVV            \\
X     @>>>      0,
\end{CD}
\]
but the existence of this bicartesian square follows from the proof of~\cite[Proposition 4.7]{Gr12}.

Finally, we will consider the morphisms in $d_J(C)$. Let $\ell \dd \Delta_2^\op \times \Delta_2^\op \setminus \{(0,0)\} \to \Delta_2^\op \times \Delta_2^\op$ be the canonical inclusion and denote by $m\dd \Delta_2^\op \times \Delta_2^\op \setminus \{(0,0)\} \to J$ the inclusion given by the picture:
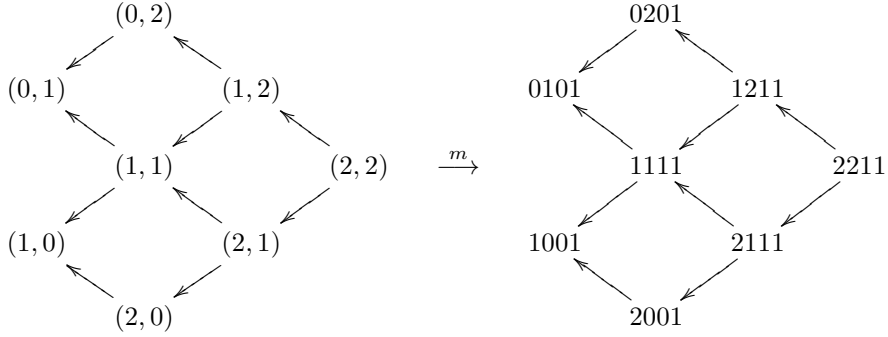
\[
\xymatrix@-1pc{
& (0,2) \ar[dl] &&& &
& 0201 \ar[dl] \\
(0,1) && (1,2) \ar[ul] \ar[dl] && &
0101 && 1211 \ar[ul] \ar[dl] \\
& (1,1) \ar[ul] \ar[dl] && (2,2) \ar[ul] \ar[dl] & \overset{m}\la &
& 1111 \ar[ul] \ar[dl] && 2211 \ar[ul] \ar[dl] \\
(1,0) && (2,1) \ar[ul] \ar[dl] && &
1001 && 2111 \ar[ul] \ar[dl] \\
& (2,0) \ar[ul] &&& &
& 2001 \ar[ul] \\
}
\]
Then, by the construction $C$ and the fact that all commutative squares in both $B$ and $C$ are bicartesian, we obtain an isomorphism $\alpha\dd \ell^*(B) \overset{\cong}\la m^*(C)$ in $\bbD(\Delta_2^\op \times \Delta_2^\op \setminus \{(0,0)\})$ whose restriction along $k$ (see~(\ref{eqn:corner_maps})) gives the canonical isomorphism
\[ k^*(\alpha)\dd k^*(B) \la k^*m^*(C) = (j_2)^*(j_1)^*(j_1)_*(j_2)_!k^*(B), \]
coming from the units/counits of the adjunctions $\big((j_1)^*,(j_1)_*\big)$ and $\big((j_2)_!,(j_2)^*\big)$. In fact, this property uniquely determines $\alpha$ since the functors $(j_1)_*$ and $(j_2)_!$ are fully faithful by \cite[Proposition 1.20]{Gr12}.

Thus, we can assume that in~(\ref{eqn:4dim}), the maps $p,q,r,s$ are placed where drawn and the compositions
\[
X \la X \oplus Y \la Z
\qquad \textrm{and} \qquad
Y \la X \oplus Y \la Z
\]
equal $u$ and $v$, respectively.

In order to compute the map $W \to X \oplus Y$ in~(\ref{eqn:4dim}), we will consider two embeddings $e_1, e_2\dd \Delta_1^\op \times \Delta_2^\op \to \Delta_2^\op \times \Delta_2^\op$ and two embeddings $f_1, f_2\dd \Delta_1^\op \times \Delta_2^\op \to J$, defined as follows. We take $e_1 = d^1 \times 1_{\Delta_2^\op}$, where $d^1$ is the face map sending $0 \mapsto 0$ and $1 \mapsto 2$ (this matches the convention in~\cite[p. 348]{Gr12}). We further put $e_2 = \sigma' e_1$, where $\sigma'\dd \Delta_2^\op \times \Delta_2^\op \to \Delta_2^\op \times \Delta_2^\op$ is given by $(n,n') \mapsto (n',n)$. Regarding the maps $f_i$, $i=1,2$, these are given by the following pictures, where $\bar i$ stands for $3-i$:
\[
\begin{CD}
(0,0) @<<< (0,1)      @<<< (0,2)  \\
@AAA       @AAA            @AAA   \\
(1,0) @<<< (1,1)      @<<< (1,2)  \\
\end{CD}
\quad\overset{f_i}\la\quad
\begin{CD}
1100  @<<< \bar{i}i01 @<<< 2201    \\
@AAA       @AAA            @AAA    \\
1110  @<<< \bar{i}i11 @<<< 2211    \\
\end{CD}
\]
Again by the construction we have isomorphisms $\beta_i\dd e_i^*(B) \overset{\cong}\la f_i^*(C)$, and these can be taken canonically in the following sense. If $k'\dd \Delta_1^\op \to \Delta_1^\op \times \Delta_2^\op$ is the map sending $n \in \Delta_1^\op$ to $(1,n+1) \in \Delta_1^\op \times \Delta_2^\op$ (i.e.\ the embedding of $\Delta_1^\op$ into $\Delta_1^\op \times \Delta_2^\op$ as the bottom right horizontal arrow), we require that
\[ (k')^*(\beta_i)\dd (k')^*(e_i)^*(B) \la (k')^*f_i^*(C) = (k')^*f_i^*(j_1)_*(j_2)_!k^*(B) \]
comes from the units/counits of the adjunctions involved. This determines $\beta_i$ uniquely.

Having these isomorphisms, we would like to conclude that the compositions
\[
W \la X \oplus Y \la X
\qquad \textrm{and} \qquad
W \la X \oplus Y \la Y
\]
in diagram~(\ref{eqn:4dim}) are equal to $x$ and $y$, respectively. This is almost true, but a little more care is needed. As a matter of the fact, the compositions
\[
\Box \overset{1_{\Delta_1^\op} \times d^1}\la \Delta_1^\op \times \Delta_2^\op \overset{e_i}\la \Delta_2^\op \times \Delta_2^\op,
\qquad
i = 1,2,
\]
mapping $\Box$ to the ``biggest'' square in $\Delta_2^\op \times \Delta_2^\op$ are not equal, but rather $e_2(1_{\Delta_1^\op} \times d^1) = e_1(1_{\Delta_1^\op} \times d^1)\sigma$, where $\sigma\dd \Box \to \Box$ acts by $(n,n') \mapsto (n',n)$. As explained in~\cite[Remark 4.13]{Gr12}, this amounts to the necessity of changing a sign at either of $x,y$, and we do so at $y$.

Finally, we deduce that the maps $W \to X \oplus Y$ and $X \oplus Y \to Z$ in~(\ref{eqn:4dim}) are $(^{\phantom{-}x}_{-y})$ and $(u,v)$, respectively. Note that the compositions
\[
X \la X \oplus Y \la X
\qquad \textrm{and} \qquad
Y \la X \oplus Y \la Y
\]
are isomorphisms by~\cite[Proposition 4.5]{Gr12}, but we know more. Thanks to our canonical identifications via $\alpha,\beta_1,\beta_2$, these compositions are even identity morphisms. Hence the claim about the shape of diagram~(\ref{eqn:4dim}) is settled.

By restricting $C$ to $\bbD(\Delta_1^\op \times \Delta_2^\op)$ along the obvious inclusion $\Delta_1^\op \times \Delta_2^\op \to J$, we get an object with diagram
\[
\begin{CD}
W @>{\left(\begin{smallmatrix}\phantom{-}x \\ -y\end{smallmatrix}\right)}>>
X \oplus Y @>>> 0   \\
@VVV @VV{(u,v)}V @VVV \\
0 @>>> Z @>{rp}>> \Sigma W
\end{CD}
\]
so that both ``small'' squares are bicartesian. Using the description of the triangulated structure from~\cite{Malt07,Gr12}, we obtain a triangle
\[
\begin{CD}
W @>{\left(\begin{smallmatrix}\phantom{-}x \\ -y\end{smallmatrix}\right)}>>
X \oplus Y @>{(u,v)}>> Z @>{rp}>> \Sigma W
\end{CD}
\]
in $\T = \bbD(e)$. In view of Definition~\ref{def:htp_cartesian}, this precisely says that $d_\Box(A)$ is homotopy cartesian.
\end{proof}

The last result connected to derivators which we need concerns homotopy colimits of countable chains. Here we first recall~\cite[Definition 1.6.4]{Nee01}.

\begin{definition} \label{def:htp_colim}
Let $\T$ be a triangulated category with countable coproducts and
\[ X_0 \overset{f_0}\la X_1 \overset{f_1}\la X_2 \overset{f_2}\la \cdots \]
be a sequence of objects and morphisms in $\T$. Then a \emph{homotopy colimit} of the chain is a cocone $(g_i\dd X_i \to \hocolim X_i)$ of the latter diagram, which is given (up to a non-canonical isomorphism uniquely) by the triangle
\[
\begin{CD}
\coprod_{i \ge 0} X_i @>{1 - \coprod f_i}>> \coprod_{i \ge 0} X_i @>{(g_i)}>> \hocolim X_i @>>> \Sigma \coprod_{i \ge 0} X_i.
\end{CD}
\]

If $u\dd X \to Y$ is a morphism between two such countable systems in $\T$, which in other words means that we have a commutative diagram of the form
\[
\begin{CD}
X_0 @>{f_0}>> X_1 @>{f_1}>> X_2 @>{f_2}>> X_3 @>{f_3}>> \cdots\phantom{,} \\
@V{u_0}VV @V{u_1}VV @V{u_2}VV @V{u_3}VV \\
Y_0 @>{g_0}>> Y_1 @>{g_1}>> Y_2 @>{g_2}>> Y_3 @>{g_3}>> \cdots, \\
\end{CD}
\]
then the axioms of triangulated categories allow us to define (non-uniquely) a map
\[ \hocolim u_i\dd \hocolim X_i \la \hocolim Y_i \]
\st the diagram with triangles defining the homotopy colimits in rows commutes:
\[
\begin{CD}
\coprod_{i \ge 0} X_i @>{1 - \coprod f_i}>> \coprod_{i \ge 0} X_i @>>> \hocolim X_i @>>> \Sigma \coprod_{i \ge 0} X_i \\
@V{\coprod u_i}VV @V{\coprod u_i}VV @V{\hocolim u_i}VV @V{\coprod \Sigma u_i}VV \\
\coprod_{i \ge 0} Y_i @>{1 - \coprod g_i}>> \coprod_{i \ge 0} Y_i @>>> \hocolim Y_i @>>> \Sigma \coprod_{i \ge 0} Y_i
\end{CD}
\]
\end{definition}

\begin{remark} \label{rem:Milnor_colim}
Some may view our terminology unfortunate as homotopy colimits usually have a different meaning in the context of derivators; see e.g.~\cite{Malt07} or~\cite[Definition 1.4]{Gr12}. Keller and Nicolas~\cite{KeNi11} use the term Milnor colimit instead, but they prove in~\cite[Proposition A.5]{KeNi11} that the Milnor homotopy colimits and derivator homotopy colimits are closely connected. More precisely, the two concepts are related essentially in the same way as bicartesian squares and homotopy cartesian squares in Proposition~\ref{prop:htp_vs_bicart}. This and the fact that we will work with Neeman's~\cite{Nee01} definition in the rest of the paper, should justify our choice of the terminology.
\end{remark}

Our result, which is a modification of~\cite[Corollary A.6]{KeNi11}, says that homotopy colimits are exact in a weak sense.

\begin{proposition} \label{prop:deriv_htp_colim}
Let $\T = \bbD(e)$, where $\bbD$ is a stable derivator \st $\bbD(I)$ has countable coproducts for each small category $I$. Suppose that $Y = (Y_i, g_i \mid i < \omega)$ is a countable direct system in $\T$ and $(u_i\dd X \to Y_i)$ a cone. Then there exists a commutative diagram in $\T$
\[
\begin{CD}
\Sigma X @= \Sigma X @= \Sigma X @= \Sigma X @= \cdots \\
@A{w_0}AA @A{w_1}AA @A{w_2}AA @A{w_3}AA \\
Z_0 @>{h_0}>> Z_1 @>{h_1}>> Z_2 @>{h_2}>> Y_3 @>{h_3}>> \cdots \\
@A{v_0}AA @A{v_1}AA @A{v_2}AA @A{v_3}AA \\
Y_0 @>{g_0}>> Y_1 @>{g_1}>> Y_2 @>{g_2}>> Y_3 @>{g_3}>> \cdots \\
@A{u_0}AA @A{u_1}AA @A{u_2}AA @A{u_3}AA \\
X @= X @= X @= X @= \cdots \\
\end{CD}
\]
\st the columns are triangles, all the squares
%
%
in the middle row are homotopy cartesian, and there is a triangle of the form
\[
\begin{CD}
X @>{\hocolim u_i}>> \hocolim Y_i @>{\hocolim v_i}>> \hocolim Z_i @>{\hocolim w_i}>> \Sigma X.
\end{CD}
\]
\end{proposition}

\begin{proof}
We can view the direct system $Y$ as an object of the diagram category $\T^\omega$ and, using~\cite[Proposition A.5(1)]{KeNi11}, lift it to an object $\tilde Y \in \bbD(\omega^\op)$ (i.e.\ $d_{\omega^\op}(\tilde Y) \cong Y$). As in~\cite{KeNi11}, we denote, for each $n < \omega$, by $n\dd e \to \omega^\op$ the functor sending $e$ to $n$. For $n = 0$ we are given the morphism $u_0\dd X \to Y_0 = 0^* \tilde Y$. Using the adjunction $(0_!,0^*)$ we get a morphism
\[ \tilde u\dd 0_!(X) \la \tilde Y \]
in $\bbD(\omega^\op)$. If we denote $\tilde X = 0_!(X)$, it is easy to see using~\cite[Lemma 1.19]{Gr12} or~\cite[Lemma A.3]{KeNi11} that the map $d_{\omega^\op}(\tilde u)\dd d_{\omega^\op}(\tilde X) \to d_{\omega^\op}(\tilde Y)$ in $\T^\omega$ is isomorphic to
\[
\begin{CD}
Y_0 @>{g_0}>> Y_1 @>{g_1}>> Y_2 @>{g_2}>> Y_3 @>{g_3}>> \cdots \\
@A{u_0}AA @A{u_1}AA @A{u_2}AA @A{u_3}AA \\
X @= X @= X @= X @= \cdots \\
\end{CD}
\]

As in~\cite{KeNi11}, we complete $\tilde u$ to a triangle in $\bbD(\omega^\op)$. Following~\cite[\S4.2]{Gr12} or~\cite{Malt07}, this amounts to constructing an object $T \in \bbD(\Delta_1^\op \times \Delta_2^\op \times \omega^\op)$, whose shape, using the partial diagram functor $d\dd \bbD(\Delta_1^\op \times \Delta_2^\op \times \omega^\op) \to \bbD(\omega^\op)^{\Delta_1 \times \Delta_2}$, is
\[
\begin{CD}
\tilde X @>{\tilde u}>> \tilde Y @>>> 0          \\
@VVV @V{\tilde v}VV @VVV                         \\
0 @>>> \tilde Z @>{\tilde w}>> \Sigma \tilde X   \\
\end{CD}
\]
and both the small squares are bicartesian. The diagram of $T$ in $\bbD(e)$, obtained by the full diagram functor $d_{\Delta_1^\op \times \Delta_2^\op \times \omega^\op}\dd \bbD(\Delta_1^\op \times \Delta_2^\op \times \omega^\op) \to \bbD(e)^{\Delta_1 \times \Delta_2 \times \omega}$, is of the form
\[
\xymatrix{
& X \ar@{=}[dd]|\hole \ar[rr]^{u_0} \ar[dl] && Y_0 \ar[dd]^(.3){g_0}|\hole \ar[dl]_{v_0} \ar[rr] && 0 \ar[dd] \ar[dl]   \\
0 \ar[dd] \ar[rr] && Z_0 \ar[dd]^(.3){h_0} \ar[rr]^(.7){w_0} && \Sigma X \ar@{=}[dd]   \\
& X \ar@{=}[dd]|\hole \ar[rr]^(.3){u_1}|\hole \ar[dl] && Y_1 \ar[dd]^(.3){g_1}|\hole \ar[dl]_{v_1} \ar[rr]|(.52)\hole && 0 \ar[dd] \ar[dl]   \\
0 \ar[dd] \ar[rr] && Z_1 \ar[dd]^(.3){h_1} \ar[rr]^(.7){w_1} && \Sigma X \ar@{=}[dd]   \\
& X \ar@{=}[dd]|(.48)\hole \ar[rr]^(.3){u_2}|\hole \ar[dl] && Y_2 \ar[dd]^(.3){g_2}|(.48)\hole \ar[dl]_{v_2} \ar[rr]|(.52)\hole && 0 \ar[dl] \ar[d]   \\
0 \ar[d] \ar[rr] && Z_2 \ar[d] \ar[rr]^(.7){w_2} && \Sigma X \ar@{=}[d] & \vdots   \\
\vdots & \vdots & \vdots & \vdots & \vdots
}
\]
Note that all the horizontal small squares are bicartesian by the construction, and so are the leftmost vertical small squares by~\cite[Proposition 4.5]{Gr12}. Using~\cite[Proposition 4.6]{Gr12}, one obtains that the middle vertical small squares, those with the diagrams of the form
\begin{equation} \label{eqn:middle_sq}
\begin{CD}
Z_i     @<{v_i}<<     Y_i       \\
@V{h_i}VV          @VV{g_i}V    \\
Z_{i+1} @<{v_{i+1}}<< Y_{i+1},  \\
\end{CD}
\end{equation}
are bicartesian as well. Hence diagram~(\ref{eqn:middle_sq}) is a homotopy cartesian square by Proposition~\ref{prop:htp_vs_bicart}.

Now the diagram from the statement of Proposition~\ref{prop:deriv_htp_colim} is obtained as the diagram of $t^*(T) \in \bbD(\Delta_3^\op \times \omega^\op)$, where $t$ is the obvious embedding $\Delta_3^\op \times \omega^\op \to \Delta_1^\op \times \Delta_2^\op \times \omega^\op$. The fact that there is a triangle of the form
\[
\begin{CD}
X @>{\hocolim u_i}>> \hocolim Y_i @>{\hocolim v_i}>> \hocolim Z_i @>{\hocolim w_i}>> \Sigma X.
\end{CD}
\]
follows from (the proof of) \cite[Proposition A.5(3)]{KeNi11}.
\end{proof}

\subsection{Algebraic triangulated categories}
\label{subsec:algebraic}

A triangulated category is called \emph{algebraic} if it can be constructed as the stable category $\underline{\C}$ of a Frobenius exact category~$\C$. We recall that an exact category is \emph{Frobenius} if it has enough projective and injective objects and the classes of projective and injective objects coincide. We refer to~\cite[Chapter I]{Hap88} for details.

Our main object of interest are compactly generated algebraic triangulated categories.

\begin{definition} \label{def:comp_gen_triang}
Let $\T$ be a category with small coproducts. An object $X \in \T$ is called \emph{compact} if for any collection $(Y_i \mid i \in I)$ of objects of $\T$, the canonical morphism
\[ \coprod \T(X,Y_i) \la \T(X, \coprod Y_i) \]
is an isomorphism. The full subcategory of all compact objects in $\T$ will be denoted by $\T^c$.

$\T$ is said to be \emph{compactly generated} if there is a set $\clS \subseteq \T$ of compact objects such that every non-zero $Y \in \T$ admits a non-zero morphism $S \to Y$ with $S \in \clS$.
\end{definition}

In order to give a description of algebraic compactly generated triangulated categories, we need the concept of dg categories introduced by Kelly~\cite{Kelly65} and Eilenberg in the 1960's. Bondal and Kapranov~\cite{BoKa89,BoKa90} later used this concept to enhance triangulated categories. Employing this nowadays standard idea, we use here the terminology and notation from~\cite{KellerDG}, where we also refer to for details.

\begin{definition} \label{def:dg}
A \emph{differential graded category} $\A$ (\emph{dg category} for short) is a category enriched over complexes of abelian groups so that the composition fulfills the Leibniz rule. That is, the morphisms sets $\A(X,Y)$ carry the structure of complexes with a differential $d$, and if $f,g$ are composable morphisms which are homogeneous of degrees $p$ and $q$, then $d(f \circ g) = d(f) \circ g + (-1)^pf \circ d(g)$.
\end{definition}

Starting with a small dg category $\A$, we adopt almost the same notations as in~\cite[\S\S 1,2 and 4]{KellerDG}:

\begin{itemize}
\item $\DifA$ will stand for the dg category of all right dg modules over $\A$; we refer to~\cite[\S1.2]{KellerDG} for details.

\item $\GraA$ is the category with the same objects, but $\GraA(X,Y)$ is the collection of all morphisms of degree zero between the underlying graded modules of $X$ and $Y$. Formally, $\GraA(X,Y)$ is the degree zero part of $\DifA(X,Y)$.

\item $\CpxA$ will denote the category of right dg modules over $\A$ with ordinary morphisms of dg modules. Formally we can write $\CpxA(X,Y) = Z^0(\DifA)(X,Y)$; see~\cite[\S2.1]{KellerDG}.

\item $\DerA$ will stand for the derived category of $\A$, which we obtain from $\CpxA$ by formally inverting all quasi-isomorphisms of dg modules.
\end{itemize}

\begin{example} \label{expl:dg_modules_ring}
Note that if $R$ is a ring and $\A$ is the dg category with one object whose endomorphism ring is $R$ with the trivial grading, then $\CpxA$ is equivalent to $\Cpx\ModR$ and $\DerA$ is triangle equivalent to $\Der\ModR$.

With this in mind, it is useful to mention a more general version of Example~\ref{expl:efficient-cpx}. Given a small dg category $\A$, we can equip $\CpxA$ with the semisplit exact structure: Conflations are defined to be those short exact sequences $0 \to X \to Y \to Z \to 0$ of dg modules which are split as short exact sequences of graded $\A$-modules. The resulting exact structure satisfies Hypothesis~\ref{hyp:exact}.
\end{example}

The following result by Keller is crucial for understanding compactly generated algebraic triangulated categories.

\begin{proposition} \label{prop:alg_comp_gen}
Let $\T$ be a triangulated category with small coproducts. Then the following are equivalent:
\begin{ii}
\item $\T$ is algebraic and compactly generated.
\item $\T$ is triangle equivalent to $\DerA$ for a small dg category $\A$.
\end{ii}
\end{proposition}

\begin{proof}
The implication (ii)~$\implies$~(i) follows from the basic properties of $\DerA$ discussed in~\cite[\S4.1]{KellerDG} (see also Remark~\ref{rem:cofibrant}). Conversely, (i)~$\implies$~(ii) has been proved in \cite[Theorem 7.5]{KrauseChicago}, which slightly improves the original argument in~\cite[Theorem 4.3]{KellerDG}.
\end{proof}

\subsection{Projective model structure for $\DerA$}
\label{subsec:model}

We will need a more explicit description of the derived category $\DerA$. First we formalize an easy observation.

\begin{lemma} \label{lem:CA_modules}
Let $\A$ be a small dg category. Then there exists a small preadditive category $\R$ \st $\CpxA$ is equivalent to $\ModRR$.
\end{lemma}

\begin{proof}
Clearly, $\CpxA$ is an abelian category and, by the proof of~\cite[Lemma 2.2]{KellerDG}, the forgetful functor $F\dd \CpxA \to \GraA$ has a left adjoint $F_\lambda\dd \GraA \to \CpxA$. Since $F$ is exact, faithful and preserves small coproducts, any set $\clP$ of compact (in the sense of Definition~\ref{def:comp_gen_triang}) projective generators for $\GraA$ is sent by $F_\lambda$ to a set $\R$ of compact projective generators of $\CpxA$.

Since $\GraA$ clearly has a set of compact projective generators, so has it $\CpxA$. Finally, we use the standard fact that if $\R$ is such a set, then the Yoneda functor
\[ Y\dd \CpxA \to \ModRR, \quad X \mapsto \CpxA(-,X)|_\R, \]
is an equivalence of categories.
\end{proof}

Now we are in a position to give a description of $\Der\A$ in terms of a model structure on $\CpxA$. We recall that a \emph{model structure} on a category is a triple $(\mathrm{Cof},\mathrm{W},\mathrm{Fib})$ of classes of morphisms, called \emph{cofibrations}, \emph{weak equivalences} and \emph{fibrations}, respectively, which satisfy certain axioms. As the axioms and other basic terminology are standard, we refer to~\cite[Definition 1.1.3]{Hov99} or \cite[Definition 7.1.3]{Hirsch03} for details.

Since $\CpxA$ is an abelian (and even a module) category, we may (and will) consider only so-called \emph{abelian model structures} on $\CpxA$. That is, we require compatibility of the model structure with the abelian exact structure on $\CpxA$ in the sense of~\cite[Definition 2.1]{Hov02} (see also~\cite[Proposition 4.2]{Hov02}):
\begin{ii}
\item cofibrations are precisely monomorphisms with cofibrant cokernels,
\item fibrations are precisely epimorphisms with fibrant kernels.
\end{ii}

Hovey~\cite[Theorem 2.2]{Hov02} proved that an abelian exact structure on a given abelian category is determined by the triple $(\C,\W,\F)$, where $\C$ is the class of cofibrant, $\W$ is the class of trivial and $\F$ the class of fibrant objects. Here, an object $X$ is called \emph{trivial} if $0 \to X$ is a weak equivalence. Moreover, he also characterized the triples $(\C,\W,\F)$ which do come from an abelian model structure as those satisfying the conditions:
\begin{ii}
\item $\W$ is closed under summands and has the 2-out-of-3 property \wrt short exact sequences,
\item $(\C,\W\cap\F)$ and $(\C\cap\W,\F)$ are functorially complete cotorsion pairs.
\end{ii}

A very nice introduction to abelian model structures is also given in~\cite[\S1]{Beck12}. In particular, a proof of the following crucial folklore result can be found there.

\begin{proposition} \label{prop:models_for_dg} \cite[Proposition 1.3.5(1)]{Beck12}
Let $\A$ be a small dg category. Then there exists a unique abelian exact structure on $\CpxA$ \st the weak equivalences are precisely the quasi-isomorphisms and every object is fibrant.

In particular, the trivially cofibrant objects are the projective ones in $\CpxA$, and the homotopy category $\Ho\CpxA$ for this model structure is equivalent to $\DerA$.
\end{proposition}

\begin{definition} \label{def:cofibrant}
From now on, an object $X \in \CpxA$ will be called \emph{cofibrant} if it is cofibrant exclusively \wrt the particular model structure from Proposition~\ref{prop:models_for_dg}. A complex $X \in \Cpx\ModR$, where $R$ is a ring, will be called \emph{cofibrant} if it is cofibrant in $\CpxA$ as in Example~\ref{expl:dg_modules_ring}.
\end{definition}

\begin{remark} \label{rem:cofibrant}
The description of cofibrant dg modules or complexes is not straightforward. Spaltenstein~\cite{Spa88} calls them $K$-projective while Keller~\cite[p. 69]{KellerDG} describes them using ``property (P)''. A very similar description to Keller's is obtained, if we combine \cite[Proposition 1.3.5(1)]{Beck12} with Proposition~\ref{prop:completeness}: A dg module is cofibrant \iff it is a summand of an $\clS$-filtered dg module, where $\clS = \{A\hat~[n] \mid A \in \A \textrm{ and } n \in \Z \}$ and $A\hat~$ are the ``free'' dg modules in the sense of~\cite[\S1.2]{KellerDG}.

In fact, now it becomes easy to see how the previously defined concepts relate together. Given a small dg category $\A$, the class $\C$ of cofibrant dg modules forms a deconstructible subcategory of $\CpxA$ by Propositions~\ref{prop:completeness}(ii) and~\ref{prop:deconstr_intersect}(i). Hence $\C$ satisfies Hypothesis~\ref{hyp:exact} by Lemma~\ref{lem:deconstr_basic}(i). It is easily checked that the abelian and the semisplit (Example~\ref{expl:dg_modules_ring}) exact structures coincide when restricted to $\C$, so that $\C$ is also a Frobenius exact category by~\cite[Lemma 2.2]{KellerDG}. The category $\DerA$ is then obtained as the stable category $\underline\C$, which can also be seen from the abelian model structure with help of~\cite[Proposition 1.1.15 and Corollary 1.1.16]{Beck12}.
\end{remark}

As a consequence of the existence of a model structure for $\DerA$, we also get the following connection to stable derivators from~\S\ref{subsec:derivators}. We also remark that analogous less general, but in our case completely relevant results can be found in~\cite[Appendice]{Malt07} and~\cite[Proposition 1.30]{Gr12}.

\begin{proposition} \label{prop:alg_comp_gen_deriv} \cite{Cis03}
Let $\T$ be a compactly generated algebraic triangulated category. Then there is a stable derivator $\bbD\dd Cat^\op \to CAT$ \st $\T = \bbD(e)$ and $\bbD(I)$ has small coproducts for each small category $I$.
\end{proposition}

\section{Compactly generated Hom-orthogonal pairs}
\label{sec:comp_gen}

\subsection{Completeness and structure results}
\label{subsec:structure}

The central concept of this text is a common generalization of $t$-structures~\cite{BBD81} and co-$t$-structures (also known as weight structures)~\cite{Bond10,MSSS11,Pauk08}, whose definitions are recalled later in Section~\ref{sec:co-t-str}.

\begin{definition} \label{def:Hom-pair} \cite[Definition 3.2]{SaSt11}
Let $\T$ be an additive category and $\clS$ be a class of objects. We denote $\clS^\perp = \{B \in \T \mid \T(S,B) = 0 \textrm{ for all } S \in \clS\}$ and dually ${^\perp \clS} = \{A \in \T \mid \T(A,S) = 0 \textrm{ for all } S \in \clS\}$. Then we call a pair $(\A,\B)$ of full subcategories a \emph{Hom-orthogonal pair} if $\A^\perp = \B$ and $\A = {^\perp \B}$. Given $\clS \subseteq \A$ \st $\B = \clS^\perp$, we say that the Hom-orthogonal pair is \emph{generated by $\clS$}.
\end{definition}

Note that given any set of objects $\clS \subseteq \T$, the pair $(\A,\B) = ({^\perp (\clS^\perp)}, \clS^\perp)$ is always a Hom-orthogonal pair generated by $\clS$.

In our context, $\T$ will usually be a triangulated category. Hom-orthogonal pairs in such categories are typically only useful if they are complete in the following sense:

\begin{definition} \label{def:complete}
Let $\T$ be a triangulated category. A Hom-orthogonal pair $(\A,\B)$ in $\T$ is called \emph{complete} if for each $X \in \T$ there is a (not necessarily unique) triangle
\[ A \la X \la B \la \Sigma A \]
with $A \in \A$ and $B \in \B$. Similarly to Definition~\ref{def:cotorsion}, we will call such a triangle an \emph{approximation triangle} for $X$.
\end{definition}

One readily sees that the map $A \to X$ is a right $\A$-approximation while $X \to B$ is a left $\B$-approximation. As discussed in~\cite[\S3]{SaSt11}, both $t$-structures and co-$t$-structures are essentially none other than complete Hom-orthogonal pairs with extra closure properties: $\Sigma\A \subseteq \A$ in the case of $t$-structures and $\Sigma\inv\A \subseteq \A$ in the case of co-$t$-structures.

\begin{remark} \label{rem:torsion_pairs}
Complete Hom-orthogonal pairs in triangulated categories have been recently called \emph{torsion pairs} or \emph{torsion theories} it the literature, for instance in~\cite[Definition 2.2]{IY08} or~\cite{AI12}. We prefer to avoid the term here since the approximation triangle is far from being unique in general, causing some well-known arguments about torsion pairs in abelian categories to fail. One should also beware that torsion pairs according to~\cite[Definition I.2.1]{BR07} are in fact precisely $t$-structures.
\end{remark}

If the triangulated category $\T$ is algebraic, the relation to cotorsion and complete cotorsion pairs is straightforward.

\begin{lemma} \label{lem:cotorsion_vs_Hom} \cite[Proposition 3.3]{SaSt11}
Let $\C$ be a Frobenius exact category and $\T = \underline\C$ the corresponding algebraic triangulated category. Then the assignment
\[ (\A,\B) \longmapsto (\A,\Sigma\B) \]
gives a bijective correspondence between cotorsion pairs in $\C$ and Hom-orthogonal pairs in $\T$.

Moreover, $(\A,\B)$ is a complete cotorsion pair in $\C$ \iff $(\A,\Sigma\B)$ is a complete Hom-orthogonal pair in $\T$.
\end{lemma}

However, the following result from~\cite{AI12}, which generalizes both \cite[Theorem III.2.3]{BR07} and \cite[Theorem 5]{Pauk11}, holds for arbitrary triangulated categories with coproducts.

\begin{proposition} \label{prop:compactly-generated-pairs} \cite[Theorem 4.3]{AI12}
Let $\T$ be triangulated category with small coproducts and let $\clS \subseteq \T$ be a set of compact objects in $\T$. Then the Hom-orthogonal pair $(\A,\B)$ generated by $\clS$ is complete.
\end{proposition}

The latter proposition provides us with a rich source of complete Hom-orthogonal pairs in triangulated categories. In order to obtain a classification, however, we need more. Starting with $\clS \subseteq \T^c$ as in Proposition~\ref{prop:compactly-generated-pairs}, we wish to know which other compact objects of $\T$ possibly belong to $\A$. In order to obtain an answer we will need to throw in more assumptions. In order to facilitate our analysis, we introduce some notations first.

\begin{notation} \label{not:ext-product}
Let $\T$ be triangulated and $\U$, $\V$ be two classes of objects. Then $\U \star \V$ stands for the class of all objects $X \in \T$ which admit a triangle of the form
\[ U \la X \la V \la \Sigma U \]
with $U \in \U$ and $V \in \V$. Similarly, we define for each $i \ge 1$
\[ \U^{\star i} = \underbrace{\U \star \cdots \star \U}_{i\ \mathrm{times}}. \]
Finally, we will use the following notation:
\begin{ii}
\item $\ext\U$ stands for the closure of $\U$ under extensions and summands;
\item $\add\U$ stands for the closure of $\U$ under finite coproducts and summands;
\item $\Add\U$ stands for the closure of $\U$ under all (possibly infinite) coproducts and under summands.
\end{ii}
\end{notation}

Note that notation above makes sense since it is a well-known consequence of the octahedral axiom that the operation $\star$ is associative. That is, $(\U \star \V) \star \W = \U \star (\V \star \W)$ for every triple $\U,\V,\W \subseteq \T$. Moreover, it is an easy observation that $X \in \ext\U$ \iff $X$ is a summand of an object in $\U^{\star i}$ for some $i \ge 1$.

Now we are able to obtain a first version of a general ``classification'' of compactly generated Hom-orthogonal pairs, which is inspired by and generalizes~\cite[Theorem A.7]{KeNi11} on compactly generated $t$-structures. It can also be viewed (under suitable assumptions) as a generalization of results on compactly generated Bousfield localizations~\cite{Nee92-2}. As in~\cite{KeNi11}, we need to assume that the triangulated category in question is at the base of a stable derivator (see~\S\ref{subsec:derivators}).

\begin{theorem} \label{thm:description_of_compacts}
Let $\T = \bbD(e)$, where $\bbD$ is a stable derivator \st $\bbD(I)$ has small coproducts for each small category $I$. Suppose that $\clS \subseteq \T$ is a set of compact objects in $\T$ and $(\A,\B)$ is the Hom-orthogonal pair in $\T$ generated by $\clS$. Then:
\begin{ii}
\item An object $X \in \T$ belongs to $\A$ \iff $X$ is a summand of a homotopy colimit of a sequence
\[ 0 = Y_0 \overset{f_0}\la Y_1 \overset{f_1}\la Y_2 \la \cdots, \]
where each $f_i$ occurs in a triangle $Y_i \overset{f_i}\to Y_{i+1} \to S_i \to \Sigma Y_i$ with $S_i \in \Add\clS$.
\item $\A \cap \T^c = \ext\clS$.
\end{ii}
\end{theorem}

As an immediate corollary we get:

\begin{corollary} \label{cor:description_of_compacts}
Let $\T = \bbD(e)$ as in Theorem~\ref{thm:description_of_compacts}. Then there is a bijective correspondence between
\begin{ii}
\item compactly generated Hom-orthogonal pairs $(\A,\B)$ in $\T$ and
\item full subcategories of $\clS \subseteq \T^c$ closed under extensions and summands.
\end{ii}
The correspondence is given by the mutually inverse assignments $(\A,\B) \mapsto \A \cap \T^c$ and $\clS \mapsto ({^\perp (\clS^\perp)}, \clS^\perp)$.
\end{corollary}

Before proving the theorem, we need a preparatory lemma. It can be viewed as a triangulated version of the Eklof Lemma (see~\cite[Lemma 3.1.2]{GT}).

\begin{lemma} \label{lem:eklof}
Let $\T$ be a triangulated category and $Z \in \T$. Suppose that we are given a countable direct system
\[ 0 = Y_0 \overset{f_0}\la Y_1 \overset{f_1}\la Y_2 \la \cdots, \]
in $\T$ where each $f_i$ belongs to a triangle $Y_i \overset{f_i}\to Y_{i+1} \to S_i \to \Sigma Y_i$ with $S_i \in {^\perp Z}$. Then $\hocolim Y_i \in {^\perp Z}$.
\end{lemma}

\begin{proof}
Lemma~5.8(2) in \cite{Bel00} used for $F = \T(-,Z)$ provides us with a short exact sequence of abelian groups
\[
0 \la {\varprojlim}^1 \T(\Sigma Y_i,Z) \la \T(\hocolim Y_i,Z) \la \varprojlim \T(Y_i,Z) \la 0,
\]
where $\varprojlim^1$ is the first derived functor of the inverse limit (see~\cite{Jensen72}). Invoking the assumption, one easily proves by induction on $i$ that $Y_i \in {^\perp Z}$ for all $i \ge 0$. In particular $\varprojlim \T(Y_i,Z) = 0$. Regarding the $\varprojlim^1$-term, the assumption on the mapping cones of the $f_i$ implies that
\[ \Hom{\T}{\Sigma f_i}{Z}\dd \Hom{\T}{\Sigma Y_{i+1}}{Z} \la \Hom{\T}{\Sigma Y_i}{Z} \]
is surjective for each $i \ge 0$. Hence the corresponding inverse system of abelian groups satisfies the Mittag-Leffler condition and ${\varprojlim}^1 \T(\Sigma Y_i,Z)$ vanishes by~\cite[Preliminaries, Proposition 13.2.2]{EGAIII-partie1}. The exact sequence above then implies the equality $\T(\hocolim Y_i, Z) = 0$, as required.
\end{proof}

\begin{proof}[Proof of Theorem~\ref{thm:description_of_compacts}]
(i) If $X$ is a summand in $\hocolim Y_i$ as in the statement of the theorem, then $X \in \A$ by Lemma~\ref{lem:eklof}. It remains to prove the converse, that is, every $X \in \A$ must be of this form.

To this end, we start as in the proof of~\cite[Theorem 4.3]{AI12}. We put $X_0 = X$ and inductively construct triangles
\[ S_i \overset{g_i}\la X_i \overset{h_i}\la X_{i+1} \la \Sigma S_i \]
so that the maps $g_i$ are right $\Add\clS$-approximations. If we denote by $u_i$ the composition $h_{i-1} \cdots h_1 h_0 \dd X \to X_i$, Proposition~\ref{prop:deriv_htp_colim} allows us to construct a diagram
\begin{equation} \label{eqn:approx}
\begin{CD}
0 @>>> \Sigma Y_1 @>{\Sigma f_1}>> \Sigma Y_2 @>{\Sigma f_2}>> \Sigma Y_3 @>{\Sigma f_3}>> \cdots \\
@AAA @AAA @AAA @AAA \\
X_0 @>{h_0}>> X_1 @>{h_1}>> X_2 @>{h_2}>> X_3 @>{h_3}>> \cdots \\
@A{1}AA @A{u_1}AA @A{u_2}AA @A{u_3}AA \\
X @= X @= X @= X @= \cdots \\
@A{0}AA @A{w_1}AA @A{w_2}AA @A{w_3}AA \\
0 @>>> Y_1 @>{f_1}>> Y_2 @>{f_2}>> Y_3 @>{f_3}>> \cdots \\
\end{CD}
\end{equation}
\st all columns are triangles, all upper squares are homotopy cartesian and there is a triangle of the form
\[
\begin{CD}
\hocolim Y_i @>{\hocolim w_i}>> X @>{\hocolim u_i}>> \hocolim X_i @>>> \Sigma \hocolim Y_i.
\end{CD}
\]

A version of the octahedral axiom~\cite[Lemma 1.4.4]{Nee01} yields triangles of the form
\[ \Sigma\inv S_i \la Y_i \overset{f_i}\la Y_{i+1} \la S_i, \]
so that the bottom row of diagram~(\ref{eqn:approx}) is a direct system as from the statement of our theorem. By the same argument as in the proof of~\cite[Proposition 4.5]{AI12} we get $\hocolim X_i \in \B$. Hence if $X \in \A$, then $\hocolim w_i$ is a split epimorphism and so $X$ has the required form.

(ii) The argument is a variation of the proof of~\cite[Lemma 2.3]{Nee92-2}, while the presentation is mostly taken from~\cite[\S5.3]{KellerDG} and~\cite[Proposition 2.2.4]{BvdB03}.

Clearly $\ext\clS \subseteq \A \cap \T^c$, so suppose conversely that $X \in \A \cap \T^c$. Consider again the morphism $X \to \hocolim X_i$ coming from diagram (\ref{eqn:approx}). Using the compactness of $X$ and the same argument as in~\cite[Lemma 2.2]{Nee92-2}, we infer that there exists $i < \omega$ so that
\[ u_i\dd X \la X_i \]
vanishes. From the construction of $u_i$ and the octahedral axiom, the object $C_i$ in the triangle
\[ C_i \overset{v}\la X \overset{u_i}\la X_i \la \Sigma C_i \] 
is easily seen to belong to $(\Add\clS)^{\star i}$. Since $u_i = 0$, the map $v$ is a split epimorphism and $X$ is a summand of $C_i \in (\Add\clS)^{\star i}$.

It suffices to prove that $X$ is also a summand of an object in $(\add\clS)^{\star i}$, as the latter is clearly a subclass of $\ext\clS$. To this end consider a split monomorphism $s\dd X \to C_i$. Since $C_i \in (\Add\clS)^{\star i}$, there exists a triangle
\[ C_{i-1} \la C_i \overset{q_i}\la S'_i \la \Sigma C_{i-1} \]
with $C_{i-1} \in (\Add\clS)^{\star(i-1)}$ and $S'_i \in \Add\clS$. As $X$ is compact, the composition $q_i s\dd X \to S'_i$ factors through an object $S''_i \in \add\clS$ and we obtain a morphism of triangles
\[
\begin{CD}
  Z_{i-1} @>>>  X  @>>>      S''_i @>>> \Sigma Z_{i-1}   \\
@V{s_{i-1}}VV @V{s}VV         @VVV             @VVV      \\
  C_{i-1} @>>> C_i @>{q_i}>> S'_i  @>>> \Sigma C_{i-1}   \\
\end{CD}
\]
Clearly both $S''_i$ and $Z_{i-1}$ are compact again. Repeating the process with the morphism $s_{i-1}\dd Z_{i-1} \to C_{i-1}$ and further inductively, we obtain a commutative diagram
\[
\begin{CD}
Z_0 @>>> Z_1 @>>> \cdots @>>> Z_{i-1} @>>> X      \\
@VVV  @V{s_1}VV     @.    @V{s_{n-1}}VV @V{s}VV   \\
0   @>>> C_1 @>>> \cdots @>>> C_{i-1} @>>> C_n    \\
\end{CD}
\]
with the cone of each of the upper maps belonging to $\add\clS$. If we consider the outer commutative square
\[
\begin{CD}
Z_0   @>{w}>>   X     \\
@VVV         @V{s}VV  \\
 0    @>>>     C_i,   \\
\end{CD}
\]
we get that $sw = 0$. But $s$ is a split monomorphism, so that $w = 0$. Hence the last map in the triangle
\[ \Sigma\inv C' \la Z_0 \overset{w}\la X \la C' \]
is a split monomorphism and by the construction and the octahedral axiom we have $C' \in (\add\clS)^{\star i}$.
\end{proof}

\subsection{Adjacent Hom-orthogonal pairs}
\label{subsec:adjacent}

An interesting phenomenon noticed by Bondarko~\cite{Bond10} and Pauksztello~\cite{Pauk08,Pauk11} is that $t$-structures and co-$t$-structures sometimes come in adjacent pairs, which can be viewed as an ``unstable'' version of recollements (or torsion torsion-free triples in the language of~\cite{BR07}). In fact, it is easy to define this notion even generally for Hom-orthogonal pairs:

\begin{definition} \label{def:adjacent}
Let $\T$ be a triangulated category and $(\X,\Y)$ and $(\A,\B)$ be two Hom-orthogonal pairs in $\T$. Then $(\X,\Y)$ is \emph{left adjacent} to $(\A,\B)$ and $(\A,\B)$ is \emph{right adjacent} to $(\X,\Y)$ if $\Y = \A$.
\end{definition}

We shall prove that for a compactly generated algebraic triangulated category a compactly generated Hom-orthogonal pair \emph{always} admits a complete right adjacent Hom-orthogonal pair. Thus our result addresses~\cite[Remark 4.5.3]{Bond10}; compare also to~\cite[Theorem 4.5.2]{Bond10}, \cite[Theorem 3.2]{Pauk08} and~\cite[Proposition 12]{Pauk11}. Another point of view at the result, in the algebraic case, can be as a generalization of the fact that a compactly generated triangulated localization induces a recollement; see~\cite[\S5.6]{Krause10}. The downside of the abstract approach here is that computing the (co)hearts as in~\cite{Bond10,Pauk08} if the Hom-orthogonal pairs correspond to general (co-)$t$-structures becomes rather difficult. We shall not address this problem here, however, but rather proceed to classification results in later sections.

\begin{theorem} \label{thm:compact_adjacent}
Let $\T$ be a compactly generated algebraic triangulated category and $(\X,\Y)$ be a Hom-orthogonal pair generated by a set of compact objects (hence complete by Proposition~\ref{prop:compactly-generated-pairs}). Then there exists a complete Hom-orthogonal pair $(\Y,\mathcal{Z})$.
\end{theorem}

\begin{proof}
Under our assumptions, $\T$ is triangle equivalent to $\DerA$ for a small dg category $\A$; see Proposition~\ref{prop:alg_comp_gen}. Now we will need the description of $\DerA$ from Proposition~\ref{prop:models_for_dg} and Remark~\ref{rem:cofibrant}. Recall that the category $\CpxA$ is essentially a module category by Lemma~\ref{lem:CA_modules}, and admits a cotorsion pair $(\C,\W)$, where $\C$ is the class of cofibrant objects in the sense of Definition~\ref{def:cofibrant} and $\W$ is the class of dg modules with zero cohomology. This cotorsion pair is generated by a set, so that it is functorially complete by Proposition~\ref{prop:completeness} and $\C$ is deconstructible by Proposition~\ref{prop:deconstr_intersect}(i). Moreover, $\C$ with the exact structure induced from $\CpxA$ is a Frobenius exact category satisfying Hypothesis~\ref{hyp:exact} and the inclusion functor $\C \to \CpxA$ preserves $\lambda$-sequences of inflations and their compositions; see Remark~\ref{rem:cofibrant} and Lemma~\ref{lem:deconstr_basic}(i). Finally, $\T$ is triangle equivalent to the stable category $\underline\C$.

Next we turn our attention to compact objects in $\T$. Put $\clS = \{A\hat~[n] \mid A \in \A \textrm{ and } n \in \Z \}$, where $A\hat~$ are the ``free'' dg modules in the sense of~\cite[\S1.2]{KellerDG}. By~\cite[Theorem 5.3]{KellerDG}, an object $S \in \underline\C$ is compact \iff $S$ is a summand in $\CpxA$ of an object of the form $S' \oplus P$, where $P$ is projective in $\CpxA$ and $S' \in \CpxA$ admits a finite $\clS$-filtration. Keller~\cite[\S2.2]{KellerDG} also showed that there is a short exact sequence in $\CpxA$ of the form
\[ 0 \la S'[-1] \la F_\lambda F(S') \la S' \la 0 \]
where $S'[-1]$ is the graded shift of $S'$ and $F_\lambda\dd \GraA \to \CpxA$ is a left adjoint to the forgetful functor $F\dd \CpxA \to \GraA$. Observe that $F_\lambda F(S')$ is necessarily compact and projective in $\CpxA$ (see the proof of Lemma~\ref{lem:CA_modules}). Hence, $S'$ has a projective resolution in $\CpxA$ of the form
\[ \cdots \la P_2 \la P_1 \la P_0 \la S' \la 0 \]
with all $P_i$ compact and projective. Consequently, both the functors $\Ext{1}{\CpxA}{S'}{-}$ and $\Ext{1}{\CpxA}{S}{-}\dd \CpxA \to \Ab$ preserve direct limits and  pure subobjects. Here we call a subobject $M$ of $N$ \emph{pure} if the inclusion $M \subseteq N$ is a direct limit of split monomorphisms.

Now suppose that $(\X,\Y)$ is a Hom-orthogonal pair in $\underline\C$ which is generated by a set $\clS$ of compact objects. Since $\C$ and $\underline\C$ have the same objects, we can view $\X$ and $\Y$ as full subcategories of $\C$, and Lemma~\ref{lem:cotorsion_vs_Hom} implies that $(\Sigma\X,\Y)$ is the cotorsion pair in $\C$ generated by $\Sigma\clS$.

Observe that Proposition~\ref{prop:completeness}(ii) implies that there is a cotorsion pair $(\Sigma\X,\Y')$ in $\CpxA$. Indeed, the description of the left hand classes of the cotorsion pairs generated by $\Sigma\clS$ in $\C$ and $\CpxA$ coincides. Then clearly $\Y = \Y' \cap \C$ and, by the discussion above, $\Y'$ is closed under direct limits and pure subobjects in $\CpxA$. Since $\CpxA$ is up to equivalence a module category over a small preadditive category, the same argument as for~\cite[Lemma 3.2.7]{GT} applies and allows us to deduce that $\Y'$ is deconstructible in $\CpxA$. Since we know that $\C$ is also deconstructible in $\CpxA$, so is $\Y = \Y' \cap \C$ by Proposition~\ref{prop:deconstr_intersect}(ii). Hence $\Y$ is also deconstructible in $\C$ by Lemma~\ref{lem:deconstr_basic}(ii).

Let us summarize: We know that $\Y = \Filt\U$ for a set of objects in $\C$, and $\Y$ certainly contains all injectives, hence projectives, and is closed under summands. Proposition~\ref{prop:completeness} then tells us that the cotorsion pair in $\C$ generated by $\U$ is of the form $(\Y,\mathcal{Z}')$ and it is complete. Thus, we obtain a complete Hom-orthogonal pair $(\Y,\mathcal{Z})$, where $\mathcal{Z} = \Sigma\mathcal{Z}'$, simply by invoking Lemma~\ref{lem:cotorsion_vs_Hom}.
\end{proof}

Knowing that compactly generated Hom-orthogonal pairs always admit right adjacent Hom-orthogonal pairs, one can ask when the converse is true.

\begin{question}[Unstable telescope conjecture] \label{ques:unstable_tc}
For which compactly generated triangulated categories $\T$ is it true that a Hom-orthogonal pair $(\X,\Y)$ admits a right adjacent pair \iff it is compactly generated?
\end{question}

If we restrict only to stable Hom-orthogonal pairs, that is those for which $\Sigma\X = \X$, this problem is known as the telescope conjecture.  The question arose in topology~\cite[Conjecture 3.4]{Bous79}, \cite[Conjecture 1.33]{Rav84} and was reformulated to a form very close to our question in~\cite[Definition 3.3.2]{HPS97}.

There certainly are algebraic triangulated categories which do not have the property asked for in Question~\ref{ques:unstable_tc}, see~\cite{KallerSmash}. However, at least the stable version is known to be true for $\T = \Der\ModR$ if $R$ is commutative noetherian~\cite{Nee92} or right hereditary~\cite{KS10}, but very little seems to be known without the restrictive assumption $\Sigma\X = \X$. A hint that the answer to Question~\ref{ques:unstable_tc} may be positive in the commutative noetherian case is given by the classification of cotilting classes in~\cite{APST12}, a relation of which to Hom-orthogonal pairs is discussed more in detail in~\cite{ASa12}.

\section{Classifying compactly generated $t$-structures and co-$t$-structures}
\label{sec:co-t-str}

After defining the necessary concepts, we will show in this section that there is often no difference between the problems of classifying compactly generated $t$-structures and classifying compactly generated co-$t$-structures. As $t$-structures are more classical, several results regarding the former problem are available in the literature. We will work out this approach for the derived category $\T = \Der\ModR$ of a commutative noetherian ring, obtaining a classification of all compactly generated co-$t$-structures in terms of so-called filtrations by supports~\cite{AJS}.

\subsection{Compactly generated $t$-structures and co-$t$-structures}
\label{subsec:com_gen}

Let us recall the concepts of a (compactly generated) $t$-structure~\cite{BBD81} and a co-$t$-structure~\cite{Bond10,MSSS11,Pauk08} on a triangulated category.

\begin{definition} \label{def:t-structure}
Let $\T$ be a triangulated category. A pair of subcategories of $\T$, $(\X, \Y)$ is called a \emph{$t$-structure} on $\T$ if it satisfies the following properties:
\begin{enumerate}
 \item[(T0)] $\X$ and $\Y$ are closed under direct summands;
 \item[(T1)] $\Sigma \X \subseteq \X$, and $\Sigma\inv\Y \subseteq \Y$;
 \item[(T2)] $\Hom{\T}{\X}{\Sigma\inv\Y} = 0$;
 \item[(T3)] For any object $Z$ of $\T$ there is a distinguished triangle
	      $$
	      X \la Z \la \Sigma\inv Y \la \Sigma X
	      $$
	      with $X \in \X$ and $Y \in \Y$.
\end{enumerate}
We call a $t$-structure \emph{compactly generated} if there is a set $\clS$ of compact objects of $\T$ such that $\Y = \clS^{\perp}$. 
\end{definition}

\begin{definition} \label{def:co-t-structure}
Let $\T$ be a triangulated category. A pair of subcategories of $\T$, $(\A, \B)$ is called a \emph{co-$t$-structure} on $\T$ if it satisfies the following properties:
\begin{enumerate}
 \item[(C0)] $\A$ and $\B$ are closed under direct summands;
 \item[(C1)] $\Sigma\inv\A \subseteq \A$, and $\Sigma \B \subseteq \B$;
 \item[(C2)] $\Hom{\T}{\Sigma\inv\A}{\B} = 0$;
 \item[(C3)] For any object $Z$ of $\T$ there is a distinguished triangle
	      $$
	      \Sigma\inv A \la Z \la B \la A
	      $$
	      with $A \in \A$ and $B \in \B$.
\end{enumerate}
We call a co-$t$-structure \emph{compactly generated} if there is a set $\clS$ of compact objects of $\T$ such that $\B = \clS^{\perp}$.
\end{definition}

\begin{remark} \label{rem:(co-)t-structure}
Note that $(\X,\Y)$ is a $t$-structure \iff $\Sigma\X \subseteq \X$ and $(\Sigma\X,\Y)$ is complete Hom-orthogonal pair. Similarly, $(\A,\B)$ is a co-$t$-structure \iff $\Sigma\inv\A \subseteq \A$ and $(\Sigma\inv\A,\B)$ is a complete Hom-orthogonal pair. The axioms as stated above are not in a minimal form. It is well known that the triangle in (T3) is unique up to a canonical isomorphism, which makes condition (T0) superfluous. It is also straightforward to check that only one inclusion in each of (T1) and (C1) is necessary. The relation between compactly generated Hom-orthogonal pairs (Definition~\ref{def:Hom-pair}) and compactly generated (co-)$t$-structures is obvious.
\end{remark}

It also turns useful for us to introduce terminology from~\cite{KV88} which captures obvious closure properties of $\A,\B,\X,\Y$ as above:

\begin{definition} \label{def:preaisles}
Let $\T$ be a triangulated category. We call a class $\U$ of objects of $\T$ a \emph{left preaisle} if $\U$ is closed under extensions, direct summands and $\Sigma\U \subseteq \U$. The class $\U$ is said to be a \emph{right preaisle} if $\U$ is closed under extensions, direct summands and $\Sigma\inv\U \subseteq \U$.
\end{definition}

Now we obtain the following as an easy consequence of our previous results.

\begin{theorem}\label{thm:preaisles}
Let $\T = \bbD(e)$, where $\bbD$ is a stable derivator \st $\bbD(I)$ has small coproducts for each small category $I$.
\begin{ii}
\item There is a bijective correspondence between compactly generated $t$-structures $(\X,\Y)$ on $\T$ and left preaisles $\U \subseteq \T^c$ of the full subcategory of compact objects, given by
\begin{eqnarray*}
(\X,\Y) &\mapsto& \X \cap \T^c, \\
\U      &\mapsto& \big({^{\perp}({\U^{\perp}})}, \Sigma\U^{\perp}\big).
\end{eqnarray*}

\item There is a bijective correspondence between compactly generated co-$t$-struc\-tures $(\A,\B)$ on $\T$ and right preaisles $\U \subseteq \T^c$, given by
\begin{eqnarray*}
(\A,\B) &\mapsto& \A \cap \T^c, \\
\U      &\mapsto& \big({^{\perp}({\U^{\perp}})}, \Sigma\inv\U^{\perp}\big).
\end{eqnarray*}
\end{ii}
\end{theorem}

\begin{proof}
This is an easy consequence of Corollary~\ref{cor:description_of_compacts}.
\end{proof}

Suppose now we would like to classify compactly generated co-$t$-structures on given $\T$. If we are lucky, we can find a nice enough triangulated category $\T'$ with small coproducts \st $(\T')^c$ is triangle equivalent to $(\T^c)^\op$. If it happens that we can classify compactly generated \emph{$t$-structures} on $\T'$, we are done. At least the first step, picking a candidate for $\T'$, is usually very easy. In the sequel, we are going to focus on one particular setting, where the situation is so favorable in that $\T^c$ is self-dual and so we can even choose $\T = \T'$.

\subsection{The tensor triangulated case}
\label{subsec:tensor_triangulated}

Often, the triangulated structure comes along with a closed symmetric monoidal product---e.g.\ some version of tensor product in algebra or the smash product in topology. The following formalism taken from~\cite{BaFa11,HPS97} will suit our needs:

\begin{definition} \label{def:tensor_triang}
Let $\T$ be a triangulated category. Then $(\T,\otimes,\unit)$ is \emph{tensor triangulated} if $\otimes$ is a closed symmetric monoidal product on $\T$ compatible with the triangulated structure in the sense of~\cite[Appendix A.2]{HPS97}. This in particular means that we have a functor
\[ \otimes\dd \T \times \T \la \T \]
together with natural isomorphisms $X \otimes (Y \otimes Z) \cong (X \otimes Y) \otimes Z$, $\unit \otimes X \cong X \cong X \otimes \unit$ and $X \otimes Y \cong Y \otimes X$ satisfying certain coherence conditions. Moreover, $- \otimes Y$ must admit a right adjoint $\RHom{}Y-$ for each $Y \in \T$, so that naturally
\[ \T(X \otimes Y, Z) \cong \T\big(X, \RHom{}YZ\big) \qquad \textrm{for all } X,Y,Z \in \T, \]
and $-\otimes-$ and $\RHom{}--$ must be triangulated functors in both variables.

Given a tensor triangulated category $\T$, an object $X \in \T$ is called \emph{rigid} (or \emph{strongly dualizable}) if, for every $Y \in \T$ the natural map $\nu\dd \RHom{}X\unit \otimes Y \to \RHom{}XY$, constructed for instance in~\cite[p. 120]{LMS86}, is an isomorphism.

Finally, $\T$ is \emph{rigidly compactly generated} if
\begin{ii}
\item $\T$ is compactly generated in the sense of Definition~\ref{def:comp_gen_triang},
\item $\unit$ is compact, and
\item every compact object is rigid.
\end{ii}
\end{definition}

\begin{remark}
One requires in~\cite[Definition A.2.1]{HPS97} that $\RHom{}-Z$ sends triangles to triangles only up to sign. This can be, however, easily fixed by changing the sing of the natural isomorphism $\RHom{}{\Sigma(-)}{Z} \to \Sigma\inv\RHom{}-Z$, which is a part of the data necessary to give a triangulated functor (see~\cite[Definition 2.1.1]{Nee01}).
\end{remark}

Despite some technicality, there are many naturally occurring examples: the derived category of a commutative ring, the stable category of representations of a finite group $G$ over a field, or the homotopy category of spectra. We refer to~\cite[Example 1.2.3]{HPS97} or~\cite[Example 1.2]{BaFa11}. We are mostly interested in the first item of the above list:

\begin{proposition} \label{prop:derived_comm}
Let $R$ be a commutative ring and denote $\T = \Der\ModR$. Then $(\T,\LOtimes R{}{},R)$, where $\LOtimes R--$ is the usual derived tensor product, is tensor triangulated and rigidly compactly generated. The right adjoint to $\LOtimes R-Y$ is the usual right derived functor $\RHom RY-$ of $\Hom RY-$.
\end{proposition}

\begin{proof}
This is well known, see also \cite[Example 5.9]{Krause10}.
\end{proof}

As far as we are concerned, the following properties of rigidly compactly generated tensor triangulated categories are important.

\begin{lemma} \label{lem:duality_compacts}
Let $(\T,\otimes,\unit)$ be tensor triangulated and rigidly compactly generated. Then an object $X \in \T$ is rigid \iff it is compact. Moreover, the functor $\RHom{}-\unit\dd \T^\op \to \T$ is induces a triangle equivalence $(\T^c)^\op \to \T^c$.
\end{lemma}

\begin{proof}
The equivalence between rigidity and compactness has been proved in \cite[Theorem 2.1.3]{HPS97}. If $X \in \T^c$, then $\RHom{}X\unit$ is compact by~\cite[III.Proposition 1.2]{LMS86} and the natural morphism $\rho\dd X \to \RHom{}{\RHom{}X\unit}\unit$ in $\T$, adjoint to the adjunction counit $\varepsilon\dd X \otimes \RHom{}X\unit \to \unit$, is an isomorphism by~\cite[Proposition III.1.3(i)]{LMS86}. Hence the restriction of $\RHom{}-\unit$ to compact objects induces a category equivalence $(\T^c)^\op \to \T^c$.
\end{proof}

Now we are able to adapt Theorem~\ref{thm:preaisles} to the tensor triangulated situation.

\begin{theorem}\label{thm:preaisles_tensor}
Let $(\T,\otimes,\unit)$ be tensor triangulated rigidly compactly generated and suppose that $\T = \bbD(e)$ for a stable derivator $\bbD$ \st $\bbD(I)$ has small coproducts for each $I$. Then there are bijective correspondences between
\begin{ii}
\item left preaisles $\U \subseteq \T^c$,
\item compactly generated $t$-structures $(\X,\Y)$ on $\T$,
\item compactly generated co-$t$-structures $(\A,\B)$ on $\T$.
\end{ii}
\end{theorem}

\begin{proof}
The bijection between (i) and (ii) is the same as in Theorem~\ref{thm:preaisles}(i). Given a full subcategory $\U \subseteq \T^c$, we denote by $\U^*$ the full subcategory given by
\[
\U^* = \{ X \mid X \cong \RHom{}Y\unit \textrm{ for some } Y \in \U \}.
\]
Clearly, the assignment $\U \mapsto \U^*$ induces an equivalence between left and right preaisles in $\T^c$. Composing this with the bijection from Theorem~\ref{thm:preaisles}(ii), we get the bijection between left preaisles in $\T^c$ and compactly generates co-$t$-structures on $\T$ given by
\[ \U \ \mapsto\  \big( {^\perp((\U^*)^\perp)} , \Sigma\inv(\U^*)^\perp \big). \qedhere \]
\end{proof}

\subsection{Classification for commutative noetherian rings}
\label{subsec:comm_noeth}

Now we aim to one of the main results which motivated this paper---the classification of compactly generated co-$t$-structures on $\Der\ModR$ for $R$ commutative noetherian. Our strategy is simple---we use the classification of compactly generated $t$-structures from~\cite{AJS} and use Theorem~\ref{thm:preaisles_tensor}. We start with a definition.

\begin{definition} \label{def:supports}
Let $R$ be a commutative ring. Then a \emph{filtration of $\Spec{R}$ by supports} is a decreasing map 
\[ \Phi \dd \Z \to 2^{\Spec{R}} \]
such that $\Phi(i)$ is a specialization closed subset of $\Spec{R}$ for each $i \in \Z$. Here we call a subset $X \subseteq \Spec{R}$ \emph{specialization closed} if $\p \in X$ implies $\q \in X$ for any $\q \in \Spec{R}$ such that $\p \subseteq \q$.
\end{definition}

Now we can formulate the classification of compactly generated $t$-structures.

\begin{proposition} \label{prop:t-structures} \cite[Theorem 3.10]{AJS}.
Let $R$ be a commutative noetherian ring. Then there is a bijective correspondence between filtrations $\Phi$ of $\Spec{R}$ by supports and compactly generated $t$-structures $(\X,\Y)$ on $\Der\ModR$. Given a $t$-structure $(\X,\Y)$, the corresponding filtration is given by
\[ \Phi(i) = \{ \p \in \Spec{R} \mid \Sigma^{-i}(R/\p) \in \X \}. \]
\end{proposition}

We also wish to describe the inverse bijection and in particular to obtain a set of compact generators for a given $t$-structure. In order to do that, it will be useful to introduce notation for certain Koszul complexes.

\begin{notation} \label{not:Koszul}
For every $\p \in \Spec R$, we fix a minimal generating set $\p = (x^\p_1, \dots, x^\p_{n_{\p}})$ and denote by $K_\p$ the Koszul complex
\[ K_\p = \bigotimes_{i=1}^{n_{\p}} (\cdots \la 0 \la R \overset{x^\p_i}\la R \la 0 \la \cdots), \]
where the cochain complexes in the tensor product are concentrated in degrees $-1$ and $0$.

We also denote by $(\Dle0,\Dge0)$ the canonical $t$-structure on $\Der\ModR$ and write as customary $\Dle n$ and $\Dge n$ for $\Sigma^{-n}\Dle0$ and $\Sigma^{-n}\Dge0$, respectively.
\end{notation}

Having adopted this notation, note that $K_\p \in \Dle0$ and $H^0(K_\p) \cong R/\p$. Now we have the following description of compactly generated $t$-structures on $\Der\ModR$.

\begin{proposition} \label{prop:koszul}
Let $R$ be a commutative noetherian ring and $\Phi \dd \mathbb Z \to 2^{\Spec{R}}$ be a filtration by supports. Then the compactly generated $t$-structure $(\X,\Y)$ corresponding to $\Phi$ as in Proposition~\ref{prop:t-structures} is generated by the set of compact objects
\[ \clS_\Phi = \{ \Sigma^{-i} K_\p \mid i \in \Z \textrm{ and } \p \in \Phi(i) \}. \]
in the sense that $\Y = \Sigma\clS_\Phi^\perp$. In particular we have the equality
\[
\Y = \{ Y \in \Der\ModR \mid \RHom{R}{\Sigma^{-i} K_\p}{Y} \in \Dge0
\textrm{ for all } i \in \Z \textrm{ and } \p \in \Phi(i) \}.
\]
\end{proposition}

\begin{proof}
The set $\clS$ generates $(\X,\Y)$ by~\cite[Corollary 3.9]{AJS}. In fact, the term ``generates'' is defined in differently in~\cite[\S1.1]{AJS}, but the two ways of generating a $t$-structure coincide by Theorem~\ref{thm:description_of_compacts}.

Using the standard isomorphisms
\[ \Der\ModR(\Sigma^{-i+j}K_\p,Y) \cong H^{-j}\big(\RHom{R}{\Sigma^{-i}K_\p}{Y}\big), \]
one easily computes that
\begin{align*}
\Y &= \{ Y \mid \Der\ModR(\Sigma^{-i+1} K_\p,Y) = 0 \textrm{ for all } i \in \Z \textrm{ and } \p \in \Phi(i) \} \\
   &= \{ Y \mid \Der\ModR(\Sigma^{-i+j} K_\p,Y) = 0 \textrm{ for all } i \in \Z, j \ge 1 \textrm{ and } \p \in \Phi(i) \} \\
   &= \{ Y \mid \RHom{R}{\Sigma^{-i} K_\p}{Y} \in \Dge0 \textrm{ for all } i \in \Z \textrm{ and } \p \in \Phi(i) \}.
   \qedhere
\end{align*}
\end{proof}

We are in a position to prove our main classification result here. At this point, it is a rather easy corollary of previous results.

\begin{theorem} \label{thm:co-t-structures}
Let $R$ be a commutative noetherian ring. Then there is a bijection between
\begin{ii}
\item filtrations $\Phi\dd \Z \to 2^\Spec{R}$ by supports, and
\item compactly generated co-$t$-structures $(\A,\B)$ on $\Der\ModR$,
\end{ii}
which is given by the assignment $\Phi \mapsto (\A_\Phi,\B_\Phi)$, where
\[
\B_\Phi =
 \{ B \in \Der\ModR \mid \LOtimes{R}{\Sigma^{-i} K_\p}{B} \in \Dle{0}
\textrm{ for all } i \in \Z \textrm{ and } \p \in \Phi(i) \}.
\]
\end{theorem}

\begin{proof}
The existence of a bijection follows from Theorem~\ref{thm:preaisles_tensor} and Proposition~\ref{prop:t-structures}. In order to understand the bijection, fix a filtration $\Phi$ and consider the set $\clS_\Phi \subseteq \Der\ModR^c$ as in Proposition~\ref{prop:koszul}. Inspecting the proof of Theorem~\ref{thm:preaisles_tensor}, the bijective correspondence is given by assigning to $\Phi$ the co-$t$-structure
\[ (\A_\Phi,\B_\Phi) =  \big( {^\perp((\clS_\Phi^*)^\perp)} , \Sigma\inv(\clS_\Phi^*)^\perp \big). \]
If we denote $\RHom RXR$ by $X^*$ for brevity, a very similar computation as in the proof of Proposition~\ref{prop:koszul} yields
\begin{align*}
\B_\Phi &= \{ B \mid \Der\ModR(\Sigma^{i-j}K_\p^*,B) = 0 \textrm{ for all } i \in \Z, j \ge 1 \textrm{ and } \p \in \Phi(i) \} \\
        &= \{ B \mid \RHom{R}{\Sigma^{i} K_\p^*}{B} \in \Dle0 \textrm{ for all } i \in \Z \textrm{ and } \p \in \Phi(i) \} \\
        &= \{ B \mid \LOtimes{R}{\Sigma^{-i} K_\p}{B} \in \Dle0 \textrm{ for all } i \in \Z \textrm{ and } \p \in \Phi(i) \}.
\end{align*}
The last isomorphism follows from the fact that $\Sigma^{i} K_\p^*$ is compact, hence rigid by Lemma~\ref{lem:duality_compacts}.
\end{proof}

\section{Perfect co-$t$-structures in the commutative noetherian case}
\label{sec:compact}

In various situations where co-$t$-structures have been studied, the emphasis has been put on co-$t$-structures on small triangulated categories, see for instance~\cite{Bond10,JoPa11}. As we have a classification of compactly generated co-$t$-structures on $\Der\ModR$, it is not very difficult at this point to classify co-$t$-structures on the skeletally small triangulated subcategory $\Der\ModR^c$ of compact objects. We recall that compact objects of $\Der\ModR$ are called \emph{perfect complexes} and they are characterized by the property that they are quasi-isomorphic to a bounded complex of finitely generated projective modules~\cite[\S6.5]{KrauseChicago}.

In fact, at least in principle this is an instance of a general approach. A well-known fact in the theory of algebraic triangulated categories says that a small algebraic triangulated category with splitting idempotents is always triangle equivalent to $\T^c$, where $\T$ is a compactly generated algebraic triangulated category (consult~\cite[Theorem 7.5(2)]{KrauseChicago}). If we know what compactly generated co-$t$-structures on $\T$ look like, we can identify the co-$t$-structures on $\T^c$ as follows.

\begin{lemma} \label{lem:restriction}
Let $\T = \bbD(e)$, where $\bbD$ is a stable derivator \st $\bbD(I)$ has small coproducts for each $I$. There is a bijection between
\begin{ii}
\item Co-$t$-structures $(\A_0,\B_0)$ in $\T^c$;
\item Compactly generated co-$t$-structures $(\A,\B)$ in $\T$ \st each $X \in \T^c$ admits an approximation triangle (as in Definition~\ref{def:co-t-structure}(C3))
\[ \Sigma^{-1}A \to X \to B \to A \]
\st $A,B \in \T^c.$
\end{ii}

The bijection is given by the assignments
\begin{eqnarray*}
  (\A_0,\B_0) &\mapsto& \big({}^\perp(\A_0^\perp), \Sigma\inv\A_0^\perp\big)  \\
  (\A, \B) &\mapsto& (\A \cap \T^c,\B \cap \T^c).
\end{eqnarray*}
\end{lemma}

\begin{proof}
Clearly, if $(\A,\B)$ is a co-$t$-structure in $\T$ as in (ii), then the restriction to $\T^c$ yields a co-$t$-structure.

Suppose on the other hand that $(\A_0,\B_0)$ is a co-$t$-structure in $\T^c$. Then $(\A,\B)$ constructed from $(\A_0,\B_0)$ as above is a compactly generated co-$t$-structure on $\T$ and $\A_0 = \A \cap \T^c$ by Theorem~\ref{thm:description_of_compacts}. Clearly also $\B_0 \subseteq \B$. It follows that the assumption from (ii) is satisfied since an approximation triangle of $X \in \T^c$ \wrt $(\A_0,\B_0)$ is automatically an approximation triangle \wrt $(\A,\B)$.
\end{proof}

Thus, our task reduces to determining which compactly generated co-$t$-structures satisfy Lemma~\ref{lem:restriction}(ii). To this end, we first establish two lemmas. In the first one, $E(M)$ stands for the injective hull of $M \in \ModR$.

\begin{lemma} \label{lem:injective_inside}
Let $R$ be a commutative noetherian ring and $\Phi\dd \Z \to 2^\Spec{R}$ be a filtration by supports. If $(\A_\Phi,\B_\Phi)$ is the corresponding compactly generated co-$t$-structure as in Theorem~\ref{thm:co-t-structures}, then for any $\p \in \Spec{R}$:
\[ \p \not\in \Phi(i) \quad \implies \quad \Sigma^{i-1} E(R/\p) \in \B_\Phi. \]
\end{lemma}

\begin{proof}
Suppose that $\p \not\in \Phi(i)$ and consider any $j \in \Z$ and $\q \in \Phi(j)$. In view of Theorem~\ref{thm:co-t-structures} we must prove that $\LOtimes{R}{\Sigma^{-j} K_\q}{\Sigma^{i-1} E(R/\p)} \in \Dle0$. This is trivial for $j<i$ since both $K_\q$ and $E(R/\p)$ belong to $\Dle0$ and
\[
\LOtimes{R}{\Sigma^{-j} K_\q}{\Sigma^{i-1} E(R/\p)} \cong
\Sigma^{i-j-1} \LOtimes{R}{K_\q}{E(R/\p)} \in \Dle0.
\]

Let us focus on the case $j \ge i$. Then $\p \not\in \Phi(j)$ and, since $\Phi(j)$ is specialization closed, one of the generators $x^\q_1, \dots, x^\q_{n_{\q}}$ of $\q$ does not belong to $\p$. This generator, say $x^\q_\ell$, acts as an isomorphism on $E(R/\p)$, so that $K_\q \otimes_R E(R/\p)$ is a contractible complex. In other words, 
\[ \LOtimes{R}{\Sigma^{-j} K_\q}{\Sigma^{i-1} E(R/\p)} = 0 \in \Dle0 \]
in this case.
\end{proof}

The other lemma is related to connected components of $\Spec{R}$. Recall that a commutative noetherian ring is called \emph{connected} if the Zariski spectrum $\Spec{R}$ is a connected topological space. This is equivalent to saying that $R$ has no idempotent elements except for $0$ and $1$; see~\cite[Exercise 2.25]{E}.

\begin{lemma} \label{lem:sp-connected}
Suppose that $R$ is a commutative noetherian ring and $\Spec{R}$ can be written as a finite disjoint union $\Spec{R} = V_1 \cup \dots \cup V_n$, where $V_n$ are specialization closed. Then all the $V_i$ are Zariski closed and they are unions of connected components of $\Spec{R}$.
\end{lemma}

\begin{proof}
Denote by $P_0$ the set of minimal prime ideals of $R$---this is a finite set by~\cite[Theorem 3.1(a)]{E}. Denoting by $V(X)$ the Zariski closure of $X \subseteq \Spec{R}$, we clearly have the inclusions
\[ V(P_0 \cap V_i) \subseteq V_i. \]
On the other hand, we have $\bigcup_{i=1}^n V(P_0 \cap V_i) = V(P_0) = \Spec{R}$ by~\cite[Corollary 10.3]{E}. It follows that $V_i = V(P_0 \cap V_i)$ is Zariski closed for each $i$.
\end{proof}

Before describing the classification of co-$t$-structures in the category of perfect complexes for $R$ commutative noetherian, we introduce some notation for the co-$t$-structures which are certainly known to exist.

\begin{notation} \label{not:known_co-t-str}
Let $R$ be a ring, $n$ be an integer, and put $\T = \Der\ModR$. Then there is a co-$t$-structure $(\K^{\ge n},\K^{\le n})$ on $\T^c$, where $\K^{\ge n}$ and $\K^{\le n}$ are the classes of perfect complexes which are isomorphic in $\T$ to bounded complexes of finitely generated projective modules concentrated in degrees $\ge n$ and $\le n$, respectively. The approximation triangles for these co-$t$-structures come from brutal truncations of complexes, see~\cite[\S1.1]{Bond10}.
\end{notation}

Now we can settle the classification of co-$t$-structures in the homotopy category of perfect complexes over a commutative noetherian ring. It turns out that these categories are very rigid it that there are no other co-$t$-structures except for those from Notation~\ref{not:known_co-t-str} and their obvious modifications if $\Spec{R}$ is disconnected.

\begin{theorem} \label{thm:perfect_co-t-str}
Let $R$ be a commutative noetherian ring and $\T = \Der\ModR$. Then there is a bijection between
\begin{ii}
\item filtrations $\Phi\dd \Z \to 2^\Spec{R}$ by supports \st each $\Phi(i)$ is a union of Zariski connected components of $\Spec{R}$, and
\item co-$t$-structures $(\A_0,\B_0)$ on $\T^c$.
\end{ii}
\end{theorem}

Before proving the theorem, we will give a more transparent interpretation of the result in the case when $R$ is connected.

\begin{corollary} \label{cor:perfect_co-t-str}
Let $R$ be a connected commutative noetherian ring and $\T = \Der\ModR$. Then the co-$t$-structures on $\T^c$ are precisely
\begin{ii}
\item the trivial ones: $(\T^c,0)$, $(0,\T^c)$;
\item (de)suspensions of the canonical one: $(\K^{\ge n},\K^{\le n})$, $n \in \Z$.
\end{ii}
\end{corollary}

\begin{proof}[Proof of Theorem~\ref{thm:perfect_co-t-str}]
Let $\Phi\dd \Z \to 2^\Spec{R}$ be a filtration by supports and $(\A_\Phi,\B_\Phi)$ the corresponding compactly generated co-$t$-structure from Theorem~\ref{thm:co-t-structures}. Our task in view of Lemma~\ref{lem:restriction} is to prove that compact objects have compact approximations \iff each $\Phi(i)$ is a union of connected components of $\Spec{R}$.

For the 'if' part, we can write $R$ uniquely as a ring product $R_1 \times \dots \times R_n$ so that each $R_i$ is connected. Denoting $\T_i = \Der{\rmod{R_i}}$, we get a canonical triangle equivalence
\[ \T^c \overset\cong\la (\T_1)^c \times \cdots \times (\T_n)^c. \]
Identifying the $(\T_i)^c$ with full subcategories of $\T^c$, consider the classes $\A_i = \A_\Phi \cap \T_i^c$ and $\B_i = \B_\Phi \cap \T_i^c$. Given our assumption on $\Phi$, it is a matter of an easy direct computation that $\A_i$ is one of $0$, $(\T_i)^c$ or $\K_i^{\ge n}$ for $n \in \Z$ (here we use the obvious modification of Notation~\ref{not:known_co-t-str}). In any case, $(\A_i,\B_i)$ is a co-$t$-structure on $(\T_i)^c$ and consequently $(\A_\Phi \cap \T^c, \B_\Phi \cap \T^c)$ is a co-$t$-structure in $\T^c$. This also explains Corollary~\ref{cor:perfect_co-t-str}.

Suppose on the other hand that $(\A_\Phi,\B_\Phi)$ satisfies the condition of Lemma~\ref{lem:restriction}(ii). We must prove that $\Phi(i)$ is a union of connected components of $\Spec{R}$ for any fixed $i \in \Z$. To this end, consider an approximation triangle
\[ \Sigma^{-1}A \la \Sigma^{i-1}R \overset{f}\la B \la A \]
with $A \in \A_\Phi \cap \T^c$ and $B \in \B_\Phi \cap \T^c$.

We claim that $\Supp{H^{1-i}(B)} = \Spec{R} \setminus \Phi(i)$, where $\Supp$ is the usual support of an $R$-module, \cite[p. 67]{E}. To see that, suppose first that $\p \not\in \Phi(i)$. Then the obvious morphism
\[ \Sigma^{i-1} R \la \Sigma^{i-1} R/\p \la \Sigma^{i-1} E(R/\p) \]
factors through $f$ by Lemma~\ref{lem:injective_inside}. Passing to the $(1-i)$-th cohomologies, we have a non-zero composition $R \to H^{1-i}(B) \to E(R/\p)$, showing that $\p \in \Supp{H^{1-i}(B)}$.

Conversely, we must show that $H^{1-i}(B)_\p = 0$ for each $\p \in \Phi(i)$. We shall prove even more: $H^{j}(B)_\p = 0$ for each $j \ge 1-i$.  By way of contradiction suppose that there is $\p \in \Phi(i)$ and $j \ge 1-i$ \st $H^j(B)_\p \cong H^j(B') \ne 0$, where $B' = B \otimes_R R_\p$. Note that $B'$ is a perfect complex over $R_\p$ and also, using the description of $\B_\Phi$ from Theorem~\ref{thm:co-t-structures}, that $B' \in \B_\Phi$. Hence we can assume that $j$ is maximal possible for given $\p$ and that $B'$ is a bounded complex of finitely generated projective $R_\p$-modules concentrated in degrees $\le j$. Since $\p \in \Phi(i)$, we have
\[ \LOtimes{R}{\Sigma^{-i} K_\p}{B'} \in \Dle0. \]

An easy computation using the fact that $K_\p$ is concentrated in degrees $\le 0$ and $B'$ in degrees $\le j$ gives
\[
0 = H^{j}(\LOtimes{R}{K_\p}{B'}) \cong H^0(K_\p) \otimes_R H^j(B')
\cong R/{\p} \otimes_R H^j(B') \cong H^j(B')/\p H^j(B').
\]
As $H^j(B')$ is a finitely generated $R_{\p}$-module, the Nakayama lemma implies $H^j(B') \cong H^j(B)_{\p} = 0$, in contradiction to our choice of $\p$ and $j$. This finishes the proof of the claim.

What we obtained is an expression of $\Spec{R}$ as the disjoint union
\[ \Spec{R} = \Phi(i) \cup \Supp{H^{1-i}(B)}, \]
where both $\Phi(i)$ and $\Supp{H^{1-i}(B)}$ are specialization closed (see~\cite[Corollary 2.7]{E} for $\Supp{H^{1-i}(B)}$). Thus, $\Phi(i)$ is a union of components of $\Spec{R}$ by Lemma~\ref{lem:sp-connected}.
\end{proof}

We conclude with two corollaries. For simplicity, we state both of them only for connected commutative noetherian rings. Their generalizations to non-connected rings are obvious.

The first one allows us to classify silting objects in $\Der\ModR$. The concept has been introduced in~\cite{KV88} and studied in detail in~\cite{AI12} recently. Recall from~\cite[Definition 4.1]{AI12} that an object $S \in \T$, where $\T$ is a triangulated category with small coproducts, is called \emph{silting} if $\T(S,\Sigma^i S) = 0$ for all $i>0$ and $\clS = \{S[i] \mid i \in \Z \}$ is a set of compact generators of $\T$ in the sense of Definition~\ref{def:comp_gen_triang}. As an easy consequence of results in~\cite{Bond10,MSSS11} we obtain:

\begin{corollary}
Let $R$ be connected commutative noetherian ring. Then $S \in \Der\ModR$ is silting \iff $S \cong \Sigma^n P$ for some $n \in \Z$ and a projective generator $P \in \rfmod{R}$.
\end{corollary}

\begin{proof}
Clearly all objects of the form $\Sigma^n P$ are silting (even tilting in the sense of~\cite{Rick89}). Conversely suppose that $S$ is silting. Then~\cite[Theorem 4.3.2]{Bond10} or~\cite[Theorem 5.5]{MSSS11} implies that $\add S = \A_0 \cap \B_0$ for some co-$t$-structure on $(\A_0,\B_0)$ on $\Der\ModR^c$. The conclusion follows from the classification in Corollary~\ref{cor:perfect_co-t-str}
\end{proof}

The second one says that non-trivial adjacent pairs of $t$-structures and co-$t$-structures in $\Der\ModR^c$ exist only under rather restrictive conditions. This is in fact closely related to the connections between silting objects and $t$-structures studied in~\cite{AI12,HKM02,KV88}. We refer to~\cite{E,MATSUMURA} for missing definitions from commutative algebra.

\begin{corollary} \label{cor:perfect_adjacent}
Let $R$ be a connected commutative noetherian ring of finite Krull dimension, let $\T = \Der\ModR$ and let $(\A_0,\B_0)$ be a non-trivial co-$t$-structure on $\T^c$ (i.e.\ $\A_0 \ne 0 \ne \B_0$). Then $(\A_0,\B_0)$ admits a left adjacent $t$-structure \iff it admits a right adjacent $t$-structure \iff $R$ is regular.
\end{corollary}

\begin{proof}
We have $(\A_0,\B_0) = (\K^{\ge n},\K^{\le n})$ for some $n \in \Z$ by Corollary~\ref{cor:perfect_co-t-str}. Hence the right adjacent $t$-structure, if it exists, must be $(\K^{\le n},\Dge n \cap \T^c)$. It is straightforward to show that this indeed is a $t$-structure \iff every finitely generated module $M \in \rfmod R$ has finite projective dimension.

It is a standard fact that the latter happens \iff $R$ is regular. Namely, if $\pd R{R/\p} < \infty$ for each $\p \in \Spec{R}$, then also $\pd{R_\p}k(\p) < \infty$, where $k(\p)$ is the residue field of $R_\p$. Thus $R_\p$ is regular for each $\p$ by the proof~\cite[Theorem 19.12]{E}, and $R$ is regular by definition~\cite[p. 157]{MATSUMURA}.

If on the other hand $R$ is regular, denote by $d$ the Krull dimension of $R$. The global dimension of each localization $R_\p$ is then bounded by $d$; see~\cite[19.9 and 18.2]{E}. Since localization at $\p$ is exact and projectivity of a finitely generated module is a local property by~\cite[Theorem 19.2]{E}, it follows that the projective dimension of each $M \in \rfmod{R}$ is bounded by $d$.

The result for left adjacent $t$-structures follows since $\T^c$ is self-dual by Lemma~\ref{lem:duality_compacts}.
\end{proof}

\bibliographystyle{abbrv}
\bibliography{co-t-str}

\begin{thebibliography}{10}

\bibitem{AI12}
T.~Aihara and O.~Iyama.
\newblock Silting mutation in triangulated categories.
\newblock {\em J. Lond. Math. Soc. (2)}, 85(3):633--668, 2012.

\bibitem{AJS}
L.~Alonso~Tarr{\'{\i}}o, A.~Jerem{\'{\i}}as~L{\'o}pez, and M.~Saor{\'{\i}}n.
\newblock Compactly generated {$t$}-structures on the derived category of a
  {N}oetherian ring.
\newblock {\em J. Algebra}, 324(3):313--346, 2010.

\bibitem{APST12}
L.~Angeleri~H{\"u}gel, D.~Posp{\'{\i}}{\v{s}}il,
  J.~{\v{S}}{\v{t}}ov{\'{\i}}{\v{c}}ek, and J.~Trlifaj.
\newblock Tilting, cotilting, and spectra of commutative {N}oetherian rings.
\newblock {\em Trans. Amer. Math. Soc.}, 366(7):3487--3517, 2014.

\bibitem{ASa12}
L.~Angeleri~H{\"u}gel and M.~Saor{\'{\i}}n.
\newblock t-{S}tructures and cotilting modules over commutative noetherian
  rings.
\newblock {\em Math. Z.}, 277(3-4):847--866, 2014.

\bibitem{AH86}
L.~Avramov and S.~Halperin.
\newblock Through the looking glass: a dictionary between rational homotopy
  theory and local algebra.
\newblock In {\em Algebra, algebraic topology and their interactions
  ({S}tockholm, 1983)}, volume 1183 of {\em Lecture Notes in Math.}, pages
  1--27. Springer, Berlin, 1986.

\bibitem{BaFa11}
P.~Balmer and G.~Favi.
\newblock Generalized tensor idempotents and the telescope conjecture.
\newblock {\em Proc. Lond. Math. Soc. (3)}, 102(6):1161--1185, 2011.

\bibitem{Beck12}
H.~Becker.
\newblock Models for singularity categories.
\newblock {\em Adv. Math.}, 254:187--232, 2014.

\bibitem{BV08}
A.~Beilinson and V.~Vologodsky.
\newblock A {DG} guide to {V}oevodsky's motives.
\newblock {\em Geom. Funct. Anal.}, 17(6):1709--1787, 2008.

\bibitem{BBD81}
A.~A. Be{\u\i}linson, J.~Bernstein, and P.~Deligne.
\newblock Faisceaux pervers.
\newblock In {\em Analysis and topology on singular spaces, {I} ({L}uminy,
  1981)}, volume 100 of {\em Ast\'erisque}, pages 5--171. Soc. Math. France,
  Paris, 1982.

\bibitem{Bel00}
A.~Beligiannis.
\newblock Relative homological algebra and purity in triangulated categories.
\newblock {\em J. Algebra}, 227(1):268--361, 2000.

\bibitem{BR07}
A.~Beligiannis and I.~Reiten.
\newblock Homological and homotopical aspects of torsion theories.
\newblock {\em Mem. Amer. Math. Soc.}, 188(883):viii+207, 2007.

\bibitem{BvdB03}
A.~Bondal and M.~van~den Bergh.
\newblock Generators and representability of functors in commutative and
  noncommutative geometry.
\newblock {\em Mosc. Math. J.}, 3(1):1--36, 258, 2003.

\bibitem{BoKa89}
A.~I. Bondal and M.~M. Kapranov.
\newblock Representable functors, {S}erre functors, and reconstructions.
\newblock {\em Izv. Akad. Nauk SSSR Ser. Mat.}, 53(6):1183--1205, 1337, 1989.

\bibitem{BoKa90}
A.~I. Bondal and M.~M. Kapranov.
\newblock Framed triangulated categories.
\newblock {\em Mat. Sb.}, 181(5):669--683, 1990.

\bibitem{BoOr95}
A.~I. Bondal and D.~O. Orlov.
\newblock Semiorthogonal decomposition for algebraic varieties.
\newblock Preprint, arXiv:alg-geom/9506012, 1995.

\bibitem{BondSurvey}
M.~V. Bondarko.
\newblock Weight structures and motives; comotives, coniveau and {C}how-weight
  spectral sequences, and mixed complexes of sheaves: a survey.
\newblock Preprint, arXiv:0903.0091v4, 2010.

\bibitem{Bond10}
M.~V. Bondarko.
\newblock Weight structures vs. {$t$}-structures; weight filtrations, spectral
  sequences, and complexes (for motives and in general).
\newblock {\em J. K-Theory}, 6(3):387--504, 2010.

\bibitem{Bous79}
A.~K. Bousfield.
\newblock The localization of spectra with respect to homology.
\newblock {\em Topology}, 18(4):257--281, 1979.

\bibitem{Bueh10}
T.~B{\"u}hler.
\newblock Exact categories.
\newblock {\em Expo. Math.}, 28(1):1--69, 2010.

\bibitem{Cis03}
D.-C. Cisinski.
\newblock Images directes cohomologiques dans les cat\'egories de mod\`eles.
\newblock {\em Ann. Math. Blaise Pascal}, 10(2):195--244, 2003.

\bibitem{CisNee08}
D.-C. Cisinski and A.~Neeman.
\newblock Additivity for derivator {$K$}-theory.
\newblock {\em Adv. Math.}, 217(4):1381--1475, 2008.

\bibitem{E}
D.~Eisenbud.
\newblock {\em Commutative algebra}, volume 150 of {\em Graduate Texts in
  Mathematics}.
\newblock Springer-Verlag, New York, 1995.
\newblock With a view toward algebraic geometry.

\bibitem{Franke96}
J.~Franke.
\newblock Uniqueness theorems for certain triangulated categories with an
  {A}dams spectral sequence, 1996.
\newblock Preprint, available at
  \url{http://www.math.uiuc.edu/K-theory/0139/Adams.pdf}.

\bibitem{GT}
R.~G{\"o}bel and J.~Trlifaj.
\newblock {\em Approximations and endomorphism algebras of modules}, volume~41
  of {\em de Gruyter Expositions in Mathematics}.
\newblock Walter de Gruyter GmbH \& Co. KG, Berlin, 2006.

\bibitem{Gr12}
M.~Groth.
\newblock Derivators, pointed derivators and stable derivators.
\newblock {\em Algebr. Geom. Topol.}, 13(1):313--374, 2013.

\bibitem{GPS14}
M.~Groth, K.~Ponto, and M.~Shulman.
\newblock Mayer-{V}ietoris sequences in stable derivators.
\newblock {\em Homology Homotopy Appl.}, 16(1):265--294, 2014.

\bibitem{EGAIII-partie1}
A.~Grothendieck.
\newblock \'{E}l\'ements de g\'eom\'etrie alg\'ebrique. {III}. \'{E}tude
  cohomologique des faisceaux coh\'erents. {I}.
\newblock {\em Inst. Hautes \'Etudes Sci. Publ. Math.}, (11):167, 1961.

\bibitem{GrothDerivators}
A.~Grothendieck.
\newblock Les d\'erivateurs.
\newblock Available at
  \url{http://www.math.jussieu.fr/~maltsin/groth/Derivateurs.html}, 1991.

\bibitem{Hap88}
D.~Happel.
\newblock {\em Triangulated categories in the representation theory of
  finite-dimensional algebras}, volume 119 of {\em London Mathematical Society
  Lecture Note Series}.
\newblock Cambridge University Press, Cambridge, 1988.

\bibitem{Heller88}
A.~Heller.
\newblock Homotopy theories.
\newblock {\em Mem. Amer. Math. Soc.}, 71(383):vi+78, 1988.

\bibitem{Hirsch03}
P.~S. Hirschhorn.
\newblock {\em Model categories and their localizations}, volume~99 of {\em
  Mathematical Surveys and Monographs}.
\newblock American Mathematical Society, Providence, RI, 2003.

\bibitem{HKM02}
M.~Hoshino, Y.~Kato, and J.-I. Miyachi.
\newblock On {$t$}-structures and torsion theories induced by compact objects.
\newblock {\em J. Pure Appl. Algebra}, 167(1):15--35, 2002.

\bibitem{Hov99}
M.~Hovey.
\newblock {\em Model categories}, volume~63 of {\em Mathematical Surveys and
  Monographs}.
\newblock American Mathematical Society, Providence, RI, 1999.

\bibitem{Hov02}
M.~Hovey.
\newblock Cotorsion pairs, model category structures, and representation
  theory.
\newblock {\em Math. Z.}, 241(3):553--592, 2002.

\bibitem{HPS97}
M.~Hovey, J.~H. Palmieri, and N.~P. Strickland.
\newblock Axiomatic stable homotopy theory.
\newblock {\em Mem. Amer. Math. Soc.}, 128(610):x+114, 1997.

\bibitem{IY08}
O.~Iyama and Y.~Yoshino.
\newblock Mutation in triangulated categories and rigid {C}ohen-{M}acaulay
  modules.
\newblock {\em Invent. Math.}, 172(1):117--168, 2008.

\bibitem{Jensen72}
C.~U. Jensen.
\newblock {\em Les foncteurs d\'eriv\'es de {$\varprojlim$} et leurs
  applications en th\'eorie des modules}.
\newblock Lecture Notes in Mathematics, Vol. 254. Springer-Verlag, Berlin,
  1972.

\bibitem{JoPa11}
P.~J{\o}rgensen and D.~Pauksztello.
\newblock The co-stability manifold of a triangulated category.
\newblock {\em Glasg. Math. J.}, 55(1):161--175, 2013.

\bibitem{Keller90}
B.~Keller.
\newblock Chain complexes and stable categories.
\newblock {\em Manuscripta Math.}, 67(4):379--417, 1990.

\bibitem{Keller91}
B.~Keller.
\newblock Derived categories and universal problems.
\newblock {\em Comm. Algebra}, 19(3):699--747, 1991.

\bibitem{KellerDG}
B.~Keller.
\newblock Deriving {DG} categories.
\newblock {\em Ann. Sci. \'Ecole Norm. Sup. (4)}, 27(1):63--102, 1994.

\bibitem{KallerSmash}
B.~Keller.
\newblock A remark on the generalized smashing conjecture.
\newblock {\em Manuscripta Math.}, 84(2):193--198, 1994.

\bibitem{KeNi11}
B.~Keller and P.~Nicol{\'a}s.
\newblock Weight structures and simple dg modules for positive dg algebras.
\newblock {\em Int. Math. Res. Not. IMRN}, (5):1028--1078, 2013.

\bibitem{KV88}
B.~Keller and D.~Vossieck.
\newblock Aisles in derived categories.
\newblock {\em Bull. Soc. Math. Belg. S\'er. A}, 40(2):239--253, 1988.
\newblock Deuxi{\`e}me Contact Franco-Belge en Alg{\`e}bre (Faulx-les-Tombes,
  1987).

\bibitem{Kelly65}
G.~M. Kelly.
\newblock Chain maps inducing zero homology maps.
\newblock {\em Proc. Cambridge Philos. Soc.}, 61:847--854, 1965.

\bibitem{KY12}
S.~Koenig and D.~Yang.
\newblock Silting objects, simple-minded collections, {$t$}-structures and
  co-{$t$}-structures for finite-dimensional algebras.
\newblock {\em Doc. Math.}, 19:403--438, 2014.

\bibitem{KrauseChicago}
H.~Krause.
\newblock Derived categories, resolutions, and {B}rown representability.
\newblock In {\em Interactions between homotopy theory and algebra}, volume 436
  of {\em Contemp. Math.}, pages 101--139. Amer. Math. Soc., Providence, RI,
  2007.

\bibitem{Krause10}
H.~Krause.
\newblock Localization theory for triangulated categories.
\newblock In {\em Triangulated categories}, volume 375 of {\em London Math.
  Soc. Lecture Note Ser.}, pages 161--235. Cambridge Univ. Press, Cambridge,
  2010.

\bibitem{KS10}
H.~Krause and J.~{\v{S}}{\v{t}}ov{\'{\i}}{\v{c}}ek.
\newblock The telescope conjecture for hereditary rings via {E}xt-orthogonal
  pairs.
\newblock {\em Adv. Math.}, 225(5):2341--2364, 2010.

\bibitem{LMS86}
L.~G. Lewis, Jr., J.~P. May, M.~Steinberger, and J.~E. McClure.
\newblock {\em Equivariant stable homotopy theory}, volume 1213 of {\em Lecture
  Notes in Mathematics}.
\newblock Springer-Verlag, Berlin, 1986.
\newblock With contributions by J. E. McClure.

\bibitem{Malt07}
G.~Maltsiniotis.
\newblock La {$K$}-th\'eorie d'un d\'erivateur triangul\'e.
\newblock In {\em Categories in algebra, geometry and mathematical physics},
  volume 431 of {\em Contemp. Math.}, pages 341--368. Amer. Math. Soc.,
  Providence, RI, 2007.

\bibitem{MATSUMURA}
H.~Matsumura.
\newblock {\em Commutative ring theory}, volume~8 of {\em Cambridge Studies in
  Advanced Mathematics}.
\newblock Cambridge University Press, Cambridge, 1986.
\newblock Translated from the Japanese by M. Reid.

\bibitem{MSSS11}
O.~Mendoza~Hern{\'a}ndez, E.~C. S{\'a}enz~Valadez, V.~Santiago~Vargas, and
  M.~J. Souto~Salorio.
\newblock Auslander-{B}uchweitz context and co-{$t$}-structures.
\newblock {\em Appl. Categ. Structures}, 21(5):417--440, 2013.

\bibitem{Nee92}
A.~Neeman.
\newblock The chromatic tower for {$D(R)$}.
\newblock {\em Topology}, 31(3):519--532, 1992.
\newblock With an appendix by Marcel B{\"o}kstedt.

\bibitem{Nee92-2}
A.~Neeman.
\newblock The connection between the {$K$}-theory localization theorem of
  {T}homason, {T}robaugh and {Y}ao and the smashing subcategories of
  {B}ousfield and {R}avenel.
\newblock {\em Ann. Sci. \'Ecole Norm. Sup. (4)}, 25(5):547--566, 1992.

\bibitem{Nee01}
A.~Neeman.
\newblock {\em Triangulated categories}, volume 148 of {\em Annals of
  Mathematics Studies}.
\newblock Princeton University Press, Princeton, NJ, 2001.

\bibitem{Pauk08}
D.~Pauksztello.
\newblock Compact corigid objects in triangulated categories and
  co-{$t$}-structures.
\newblock {\em Cent. Eur. J. Math.}, 6(1):25--42, 2008.

\bibitem{Pauk11}
D.~Pauksztello.
\newblock A note on compactly generated co-{$t$}-structures.
\newblock {\em Comm. Algebra}, 40(2):386--394, 2012.

\bibitem{Rav84}
D.~C. Ravenel.
\newblock Localization with respect to certain periodic homology theories.
\newblock {\em Amer. J. Math.}, 106(2):351--414, 1984.

\bibitem{Rick89}
J.~Rickard.
\newblock Morita theory for derived categories.
\newblock {\em J. London Math. Soc. (2)}, 39(3):436--456, 1989.

\bibitem{SaSt11}
M.~Saor{\'{\i}}n and J.~{\v{S}}{\v{t}}ov{\'{\i}}{\v{c}}ek.
\newblock On exact categories and applications to triangulated adjoints and
  model structures.
\newblock {\em Adv. Math.}, 228(2):968--1007, 2011.

\bibitem{Spa88}
N.~Spaltenstein.
\newblock Resolutions of unbounded complexes.
\newblock {\em Compositio Math.}, 65(2):121--154, 1988.

\bibitem{Sto11}
J.~{\v{S}}{\v{t}}ov{\'{\i}}{\v{c}}ek.
\newblock Deconstructibility and the {H}ill lemma in {G}rothendieck categories.
\newblock {\em Forum Math.}, 25(1):193--219, 2013.

\end{thebibliography}

\end{document}